\documentclass[11pt,letterpaper]{amsart}
\usepackage{amsmath,amssymb,amsthm,mathtools}
\usepackage{booktabs}
\usepackage{array}
\usepackage{mathrsfs}
\usepackage{microtype}
\usepackage{xcolor}
\usepackage[colorlinks=true,linkcolor=blue!60!black,citecolor=blue!60!black,urlcolor=blue!60!black]{hyperref}
\hypersetup{
  pdftitle={Finite-Defect Rigidity and the Minimum Spherical 4-Design on the Two-Sphere},
  pdfauthor={Shalender Singh and Vishnu Priya Singh Parmar},
  pdfsubject={Finite-defect structure, corank-one residue formulae, and the exact value N4(S2)=12},
  pdfkeywords={spherical designs, finite defect, Naimark complement, Gale duality, Schoenberg coefficients, Cayley-Bacharach}
}

\newtheorem{theorem}{Theorem}[section]
\newtheorem{lemma}[theorem]{Lemma}
\newtheorem{proposition}[theorem]{Proposition}
\newtheorem{corollary}[theorem]{Corollary}
\newtheorem{fact}[theorem]{Fact}
\theoremstyle{definition}
\newtheorem{definition}[theorem]{Definition}
\theoremstyle{remark}
\newtheorem{remark}[theorem]{Remark}
\newtheorem{example}[theorem]{Example}
\numberwithin{equation}{section}

\newcommand{\R}{\mathbb{R}}
\newcommand{\C}{\mathbb{C}}
\newcommand{\Q}{\mathbb{Q}}
\newcommand{\Z}{\mathbb{Z}}
\newcommand{\Sph}[1]{\mathbb{S}^{#1}}
\newcommand{\PP}{\mathbb{P}}
\newcommand{\Pol}{\mathcal{P}}
\newcommand{\Harm}{\mathcal{H}}
\newcommand{\FF}{\mathcal{F}}
\newcommand{\LL}{\mathcal{L}}
\newcommand{\Gal}{\operatorname{Gal}}

\newcommand{\Herm}{\operatorname{Herm}_2}
\newcommand{\rank}{\operatorname{rank}}
\newcommand{\ran}{\operatorname{ran}}
\newcommand{\tr}{\operatorname{tr}}
\newcommand{\Soc}{\operatorname{Soc}}
\renewcommand{\Re}{\operatorname{Re}}
\renewcommand{\Im}{\operatorname{Im}}
\newcommand{\dd}{\,\mathrm{d}}
\newcommand{\eps}{\varepsilon}
\newcommand{\vp}{v_p}

\title[Finite-defect rigidity and the minimum $4$-design]{Finite-Defect Rigidity and the Minimum Spherical $4$-Design on the Two-Sphere}
\author{Shalender Singh}
\author{Vishnu Priya Singh Parmar}
\date{September 12, 2026}

\begin{document}

\begin{abstract}
We prove that every equal-weight spherical $4$-design on $\Sph{2}$ has at least twelve points.  Since the regular icosahedron is a spherical $5$-design, this determines the exact minimum
\[
 N_4(\Sph{2})=12;
\]
equivalently, no such design has $9$, $10$ or $11$ points.

The proof is part of a finite-defect theory.  If a spherical $2m$-design has corank $c=N-\dim\Pol_m$, its Naimark complement consists of unit vectors $u_x\in\Sph{c-1}$ forming a spherical $2$-design and satisfying the exact coupling
\[
 u_x\cdot u_y=-\frac{K_m^{(d)}(x\cdot y)}{c}\qquad(x\ne y).
\]
This gives a pairwise kernel bound, an antipodal lower bound, Cayley--Bacharach information, and uniform lower bounds for multiplicative-relation spaces.  In corank one the design splits into two equal spherical $m$-designs.  We derive a residue formula for its signed Schoenberg coefficients; the coefficient of degree $m+3$ is negative exactly when $3\le d\le m+1$, excluding corank one throughout that range.  At strengths four and six the only examples in any dimension are the regular hexagon and octagon, respectively.

In corank two the complement is a circle Gale frame.  Multiplication by its phase forces at least $d-1$ linear--quadratic aliases and yields exact norm and socle identities in every dimension.  For eleven nodes on $\Sph{2}$, two aliases produce a real harmonic cubic and a Hermitian quartic matrix.  A matrix-valued Cayley--Bacharach argument eliminates the generic branch; the exceptional branch reduces to a Pauli normal form and contradicts the second moments.  In dimensions $d\ge4$ the corank-two problem remains open; we identify a forced quadratic socle as the obstruction to extending the present argument.
\end{abstract}

\maketitle

\section{Introduction}\label{sec:intro}

\noindent\textbf{Main result and scope.}
A finite set $X\subset\Sph{d-1}$ is an \emph{equal-weight spherical $t$-design} if
\[
 \frac1{|X|}\sum_{x\in X}p(x)=\int_{\Sph{d-1}}p\dd\omega
\]
for every polynomial $p$ of degree at most $t$, where $\omega$ is the rotation-invariant probability measure.  Spherical designs were introduced by Delsarte, Goethals and Seidel \cite{DGS}.  Throughout, ``design'' means equal-weight design with distinct nodes.  All notation and the definitions used in the statements below (evaluation spaces, corank, kernels, Naimark complement and Gale vectors, Cayley--Bacharach property, Hilbert function and socle, Schoenberg coefficients) are collected in Section~\ref{sec:kernels}, which also contains a table of symbols.

For $4$-designs on $\Sph{2}$ the Fisher bound is $9$.  The tight value $9$ is impossible \cite{BD1}, while the twelve vertices of the regular icosahedron form a spherical $5$-design.  Hardin and Sloane \cite[Theorem~4]{HS92} constructed spherical $4$-designs in $\R^3$ with
\[
 N=12,\ 14\quad\text{and every }N\ge16
\]
and conjectured \cite[Conjecture~2]{HS92} that these are the only cardinalities.  The conjecture is restated in \cite{HS96} and, in the notation $M_2(4)=12$, in \cite{Bajnok98}.  The values $10,11,13,15$ were left open.  We settle the first two.

\begin{theorem}\label{thm:main}
There is no equal-weight spherical $4$-design on $\Sph{2}$ with $9$, $10$ or $11$ points.  Hence
\[
 N_4(\Sph{2})=12,
\]
and the minimum is attained by the regular icosahedron.
\end{theorem}

The cases $N=13$ and $N=15$ remain open.  The proof of Theorem~\ref{thm:main} is embedded in a broader theory organised by the defect of the degree-$m$ evaluation frame.

Let
\[
 D_m=D_m(d):=\dim\Pol_m(\Sph{d-1})
 =\binom{d+m-1}{m}+\binom{d+m-2}{m-1},
\]
and call $c=N-D_m$ the \emph{corank} of a spherical $2m$-design with $N$ nodes (Definition~\ref{def:corank}).  Normalised evaluation is an isometry; its orthogonal complement is a Naimark complement in the sense of frame theory \cite{CFMPS13} and plays the role of a Gale dual (Section~\ref{sec:def-frames}).  A related Gale-dual viewpoint for graphical designs appears in \cite{BT23}.  The first theorem describes every positive corank.

\begin{theorem}[Finite-defect structure]\label{thm:intro-finite-defect}
Let $X\subset\Sph{d-1}$ be a spherical $2m$-design of corank $c\ge1$, and let $K_m^{(d)}$ be the reproducing kernel of $\Pol_m$.  There are unit vectors $u_x\in\Sph{c-1}$ which, as a multiset, form a spherical $2$-design and satisfy
\[
 \sum_{x\in X}p(x)u_x=0\quad(\deg p\le m),
 \qquad
 u_x\cdot u_y=-\frac{K_m^{(d)}(x\cdot y)}c\quad(x\ne y).
\]
Consequently $|K_m^{(d)}(x\cdot y)|\le c$ for distinct nodes.  If $X$ contains an antipodal pair, then
\[
 |X|\ge2\binom{d+m-1}{m},
\]
which is the Delsarte--Goethals--Seidel lower bound for $(2m+1)$-designs \cite[Theorem~5.12]{DGS}.  The projective lift of $X$ is Cayley--Bacharach in degree $m$, and when $c\le2$ its evaluation space saturates in degree $m+1$.
\end{theorem}

This is Theorem~\ref{thm:finite-defect} and Proposition~\ref{prop:cb-hilbert}.  A further relation theorem, Proposition~\ref{prop:relation-count}, gives a lower bound $dc-h_m$ for a natural space of bilinear equations coupling the nodes to their Gale vectors.  Here and below $h_m=\dim\Harm_m(\R^d)$ denotes the dimension of the space of harmonic polynomials of degree $m$.

\medskip
\noindent\textbf{Corank one.}
When $c=1$, the Gale sphere is $\Sph{0}$ and the complement is a balanced sign vector.  The design therefore splits into two equal spherical $m$-designs (Theorem~\ref{thm:split}).  At strength four a spectral-integrality argument leaves only the regular hexagon (Theorem~\ref{thm:strength4}), excluding ten nodes on $\Sph{2}$; the residue formula below gives an independent second proof of that exclusion.  At strength six the only example is the regular octagon (Theorem~\ref{thm:strength6}).  Galois constancy of the forced distance distribution and a $p$-adic Newton-polygon argument give additional all-strength nonexistence families in Section~\ref{sec:higher}.

The new uniform ingredient is a closed residue formula.  If $\eps_x\in\{\pm1\}$ is the corank-one sign and $P_\ell$ is the normalised zonal polynomial, define
\[
 A_{\ell,1}=\frac1{N^2}\sum_{x,y\in X}\eps_x\eps_yP_\ell(x\cdot y)\ge0.
\]
Theorem~\ref{thm:residue} expresses every $A_{\ell,1}$ as a coefficient at infinity.  In particular, $A_{m+3,1}<0$ exactly for $3\le d\le m+1$.

\begin{theorem}[Corank-one nonexistence range]\label{thm:intro-residue}
Let $m\ge2$ and $3\le d\le m+1$.  No spherical $2m$-design on $\Sph{d-1}$ has $D_m+1$ points.  In particular, for every $m\ge2$ there is no spherical $2m$-design on $\Sph{2}$ with $(m+1)^2+1$ points.
\end{theorem}

The known nonexistence results for cardinalities at or near the Fisher bound are the classification of tight designs, that is, corank zero \cite{BB09,BD1,BD2,BMV04}, and the improved linear-programming bounds and the exclusions of specific parameters they yield \cite{BDN99,Yudin97}.  Theorem~\ref{thm:intro-residue} concerns corank one and is uniform in the strength; its proof is the closed residue formula of Theorem~\ref{thm:residue}, not a linear programme.  Together with Theorems~\ref{thm:strength4} and \ref{thm:strength6} it excludes corank one whenever $m\le3$ or $d\le m+1$.

\medskip
\noindent\textbf{Corank two.}
For $c=2$ the Gale vectors form a unit circle frame, encoded by phases $\zeta_x$.  Pointwise multiplication by $\bar\zeta$ produces a space of linear--quadratic aliases of dimension at least $d-1$ (Theorem~\ref{thm:alias}).  Theorem~\ref{thm:corank-two-identities} gives an exact isometry for these aliases and a decomposition of $\zeta^2$; Proposition~\ref{prop:socle-alias} describes the quadratic socle through the aliases and recovers the forced-socle bound of Theorem~\ref{thm:forced-quadratic-socle}.

\begin{theorem}[Corank-two structure and low-dimensional classification]\label{thm:intro-corank-two}
Let $X\subset\Sph{d-1}$ be a spherical $4$-design with
\[
 |X|=D_2(d)+2=\frac{d(d+3)}2+2.
\]
Then the cubic evaluation map is onto, the projective lift of $X$ is Cayley--Bacharach in degree two, and the Artinian reduction by the homogenising coordinate has Hilbert vector
\[
 \left(1,d,\frac{(d-1)(d+2)}2,2\right).
\]
Its degree-two socle is the common orthogonal complement of the alias quadratics and their conjugates, and has dimension at least
\[
 \max\left\{0,\frac{d^2-3d-2}{2}\right\}.
\]
For $d=2$, $X$ is a regular heptagon.  For $d=3$, no such design exists.  The existence problem for $d\ge4$ remains open.
\end{theorem}

For eleven nodes on $\Sph{2}$, two aliases yield a nonzero real harmonic cubic.  The associated Hermitian quartic matrix has determinant $-r^2R^2$.  A generic scalar compression gives a $(3,4)$ complete intersection, where a matrix-valued Cayley--Bacharach argument forces an impossible rank-two factorisation.  In the exceptional branch the matrix factors as $R$ times a Hermitian pencil of determinant $-r^2$; a Pauli normal form and harmonic-spinor classification then imply an impossible comparison $11/3\ge11/2$.  Sections~\ref{sec:circle}--\ref{sec:moment} carry out this exact argument.

\medskip
\noindent\textbf{How Theorem~\ref{thm:main} is proved.}
The three cardinalities are excluded by three different mechanisms.  Nine points is the tight case; Lemma~\ref{lem:nine} gives a short kernel proof.  Ten points is corank one; Theorem~\ref{thm:strength4} excludes it in every dimension by the spectral argument, and Theorem~\ref{thm:sign-obstruction} with $(d,m)=(3,2)$ excludes it a second time by the residue formula.  Eleven points is corank two; Sections~\ref{sec:circle}--\ref{sec:moment} prove Theorem~\ref{thm:eleven} by contradiction through the Gale circle, two aliases, the harmonic cubic $R$, and the dichotomy between the generic branch (Proposition~\ref{prop:generic}) and the exceptional branch (Proposition~\ref{prop:exceptional}).  Section~\ref{sec:minimum} assembles these into Theorem~\ref{thm:main}.

\medskip
\noindent\textbf{Organisation of the paper.}
Section~\ref{sec:kernels} collects notation, definitions and normalisations, with a table of symbols.  Section~\ref{sec:finite-defect} proves the finite-defect structure theorem at every corank.  Sections~\ref{sec:split}--\ref{sec:residue} treat corank one: the splitting theorem, the classification at strength four, the Galois and Newton-polygon obstructions with the classification at strength six, and the residue formula with its uniform nonexistence range.  Section~\ref{sec:all-dimensional-corank-two} develops the corank-two structure in every dimension.  Sections~\ref{sec:circle}--\ref{sec:moment} prove that no eleven-point $4$-design exists on $\Sph{2}$; the standing assumptions for that part are stated at the beginning of Section~\ref{sec:circle}.  Section~\ref{sec:minimum} proves Theorem~\ref{thm:main}, and Section~\ref{sec:further} records a stability corollary and open problems.  Appendix~\ref{app:identities} contains the coefficient extraction behind Theorem~\ref{thm:residue}, Appendix~\ref{app:certificates} the two finite computational certificates used in Section~\ref{sec:higher}, and Appendix~\ref{app:repro} describes the ancillary verification scripts.

\medskip
\noindent\textbf{Related work.}
Tight designs and their association schemes are treated in \cite{BB09,BBTZ17,BD1,BD2,BMV04}.  The word ``rigidity'' in our title refers to the structure forced on a design by a small corank; it is not the local rigidity of designs under deformation studied by Bannai \cite{Bannai87} in connection with Hong's theorem \cite{Hong82}.  Existence for all sufficiently large cardinalities is due to Seymour and Zaslavsky \cite{SZ84}, with optimal asymptotic bounds by Bondarenko, Radchenko and Viazovska \cite{BRV13}.  Numerical designs are tabulated in \cite{GP11,HS96,Womersley18}.  Lower bounds beyond the basic linear programme include \cite{BDN99,Yudin97}.  Reznick's memoir \cite{Reznick92} treats related Hilbert identities; for fixed-strength designs see Dillon \cite{Dillon26}.  The kernel $K_m^{(d)}$ is Levenshtein's adjacent polynomial \cite{Levenshtein98}, equivalently a Jacobi polynomial with parameters $((d-1)/2,(d-3)/2)$ \cite{DGS,Szego}.

\section{Notation, definitions and standing conventions}\label{sec:kernels}

This section collects every definition and normalisation used later; a table of symbols closes it.  Readers familiar with spherical designs may skim Sections~\ref{sec:def-designs}--\ref{sec:def-kernels} and consult Sections~\ref{sec:def-frames}--\ref{sec:def-schoenberg} for the frame-theoretic and algebraic vocabulary.

\subsection{Designs, evaluation spaces and corank}\label{sec:def-designs}

Throughout, $d\ge2$, $\Sph{d-1}=\{x\in\R^d:|x|=1\}$, and $\omega$ is the rotation-invariant probability measure on $\Sph{d-1}$.  Let $\Pol_j=\Pol_j(\Sph{d-1})$ denote the space of restrictions to $\Sph{d-1}$ of real polynomials of degree at most $j$, and $\Pol_j^\C$ its complexification.

\begin{definition}[Spherical design]\label{def:design}
A finite nonempty set $X\subset\Sph{d-1}$ is an \emph{equal-weight spherical $t$-design} if
\[
 \frac1{|X|}\sum_{x\in X}p(x)=\int_{\Sph{d-1}}p\dd\omega\qquad(p\in\Pol_t).
\]
Throughout, ``design'' means equal-weight design with distinct nodes; only Corollary~\ref{cor:stability} allows repeated nodes, and says so.  The elements of $X$ are called \emph{nodes}, and $N=|X|$.
\end{definition}

\begin{definition}[Sampled inner product, evaluation spaces]\label{def:evaluation}
For a finite $X\subset\Sph{d-1}$ equip $\C^X$ with the sampled inner product and norm
\[
 \langle f,g\rangle_X=\frac1N\sum_{x\in X}f(x)\overline{g(x)},
 \qquad
 \|f\|_X^2=\langle f,f\rangle_X .
\]
The \emph{degree-$j$ evaluation space} of $X$ is
\[
 \mathcal E_j(X)=\bigl\{(p(x))_{x\in X}:p\in\Pol_j^\C\bigr\}\subset\C^X ,
\]
and $\mathcal E_j(X)^\perp$ denotes its orthogonal complement in $\C^X$.  We write $\Harm_j^\C(X)$ for the evaluations of complex harmonic polynomials of degree $j$ (defined in Section~\ref{sec:def-kernels}), so that $\mathcal E_j(X)=\Harm_0^\C(X)+\cdots+\Harm_j^\C(X)$.
\end{definition}

\begin{definition}[Design exactness]\label{def:exactness}
If $X$ is a spherical $2m$-design, then for $p,q\in\Pol_m^\C$ the product $p\bar q$ has degree at most $2m$, so
\[
 \langle p,q\rangle_X=\int_{\Sph{d-1}}p\bar q\dd\omega=:\langle p,q\rangle_{L^2}.
\]
We refer to this as \emph{design exactness through degree $2m$}.  In particular evaluation $\Pol_m^\C\to\mathcal E_m(X)$ is an isometry, hence injective, and $\mathcal E_m(X)=\Harm_0^\C(X)\oplus\cdots\oplus\Harm_m^\C(X)$ is an orthogonal direct sum.
\end{definition}

\begin{definition}[Corank]\label{def:corank}
Put
\[
 D_m=D_m(d):=\dim\Pol_m(\Sph{d-1})=\binom{d+m-1}{m}+\binom{d+m-2}{m-1}.
\]
Since evaluation is injective on $\Pol_m$, every spherical $2m$-design satisfies $N\ge D_m$, the Fisher-type bound of \cite[Theorem~5.11]{DGS}.  The \emph{corank} of a spherical $2m$-design $X$ is
\[
 c=c(X):=N-D_m\ge0 .
\]
Corank zero is tightness in the sense of \cite{BD1,DGS}; this paper concerns positive corank, and mainly $c\in\{1,2\}$.
\end{definition}

\subsection{Harmonic polynomials, zonal polynomials and the reproducing kernel}\label{sec:def-kernels}

Let $\Harm_j(\R^d)$ be the real homogeneous harmonic polynomials of degree $j$, let $h_j=\dim\Harm_j$, and write $\Harm_j^\C$ for the complexification.  Restricted to $\Sph{d-1}$,
\[
 \Pol_m=\bigoplus_{j=0}^m\Harm_j,
 \qquad
 h_j=D_j-D_{j-1},
\]
so that $h_0=1$, $h_1=d$ and $h_j=\binom{d+j-1}{j}-\binom{d+j-3}{j-2}$ for $j\ge2$.
Let $Q_j^{(d)}$ be the zonal Gegenbauer polynomial of degree $j$, normalised by $Q_j^{(d)}(1)=1$ (for $d=2$, $Q_j^{(2)}$ is the Chebyshev polynomial $T_j$).  For an orthonormal basis $(\phi_\alpha)$ of $\Harm_j$ the addition formula reads
\[
 \sum_\alpha\phi_\alpha(x)\phi_\alpha(y)=h_jQ_j^{(d)}(x\cdot y).
\]
The reproducing kernel of $\Pol_m$ is
\begin{equation}\label{eq:kernel}
 K_m^{(d)}(t)=\sum_{j=0}^mh_jQ_j^{(d)}(t),
 \qquad K_m^{(d)}(1)=D_m .
\end{equation}
For $m=2$,
\begin{equation}\label{eq:K2}
 K_2^{(d)}(t)=1+dt+\frac{d+2}{2}(dt^2-1)
 =\frac d2\bigl((d+2)t^2+2t-1\bigr),
\end{equation}
and in $d=3$,
\begin{equation}\label{eq:K2d3}
 K(t):=K_2^{(3)}(t)=\frac32(5t^2+2t-1).
\end{equation}

Let $\sigma_d$ be the pushforward of $\omega$ under $y\mapsto x\cdot y$, independent of $x$, and write $\dd\sigma_d(t)=w_d(t)\dd t$.  Its moments are
\begin{equation}\label{eq:moments}
 \mu_j:=\int_{-1}^1t^j\dd\sigma_d(t)
 =\begin{cases}
 \dfrac{(j-1)!!}{d(d+2)\cdots(d+j-2)},&j\text{ even},\\[1.5ex]
 0,&j\text{ odd}.
 \end{cases}
\end{equation}
The zonal polynomials are orthogonal with
\[
 \int_{-1}^1Q_j^{(d)}Q_{j'}^{(d)}\dd\sigma_d=\frac{\delta_{jj'}}{h_j},
\]
and the kernel reproduces at $1$:
\begin{equation}\label{eq:reproduce-one}
\begin{aligned}
 \int_{-1}^1K_m^{(d)}(u)p(u)\dd\sigma_d(u)&=p(1) &&(\deg p\le m),\\
 \int K_m^{(d)}\dd\sigma_d&=1,
 &\qquad \int (K_m^{(d)})^2\dd\sigma_d&=D_m.
\end{aligned}
\end{equation}
Two consequences used repeatedly: for a $2m$-design, $\frac1N\sum_{y\in X}p(x\cdot y)=\int p\dd\sigma_d$ whenever $\deg p\le2m$; and the Gram matrix $\bigl(K_m^{(d)}(x\cdot y)\bigr)_{x,y}$ is positive semidefinite, being the Gram matrix of the vectors $(\phi_\alpha(x))_\alpha$ over an orthonormal basis of $\Pol_m$.

\subsection{Frames, the Naimark complement and Gale vectors}\label{sec:def-frames}

Fix an orthonormal basis $(\phi_\alpha)_{\alpha=1}^{D_m}$ of $\Pol_m$ (with respect to $\omega$) and let $X$ be a spherical $2m$-design.  The \emph{normalised evaluation matrix} is
\[
 E\in\R^{N\times D_m},\qquad E_{x\alpha}=N^{-1/2}\phi_\alpha(x).
\]
Design exactness says $E^{T}E=I_{D_m}$: the rows $v_x=N^{-1/2}(\phi_\alpha(x))_\alpha$ form a \emph{Parseval frame} of $\R^{D_m}$, and $P=EE^{T}$ is the orthogonal projector of $\C^X$ onto $\mathcal E_m(X)$, with entries $P_{xy}=N^{-1}K_m^{(d)}(x\cdot y)$.

Following \cite{CFMPS13}, a \emph{Naimark complement} of the Parseval frame $(v_x)$ is a Parseval frame $(v_x')$ in $\R^{c}$, $c=N-D_m$, obtained as the rows of any matrix $V\in\R^{N\times c}$ with orthonormal columns spanning $\ran(I-P)=\mathcal E_m(X)^\perp$.  The columns of $V$ are exactly the linear relations $\sum_xw_xv_x=0$ among the evaluation vectors, so $(v_x')$ is a \emph{Gale transform} of the configuration $(v_x)$ in the sense of convex geometry; Gale duality is used in the same way for graphical designs in \cite{BT23}.  We call the rescaled rows
\[
 u_x:=\sqrt{N/c}\;v_x'\in\R^c
\]
the \emph{Gale vectors} of $X$; Theorem~\ref{thm:finite-defect} shows that they are unit vectors.  They are unique up to a common orthogonal transformation of $\R^c$.

\subsection{Projective lift, Cayley--Bacharach property, Hilbert function and socle}\label{sec:def-algebra}

\begin{definition}[Projective lift and Hilbert function]\label{def:lift}
The \emph{projective lift} of $X$ is $\widehat X=\{[1:x]:x\in X\}\subset\PP^d_\C$, a set of $N$ distinct points.  Its homogeneous coordinate ring is $R_X=\C[z_0,\ldots,z_d]/I(\widehat X)$, with Hilbert function $H_X(j)=\dim_\C(R_X)_j$.  Dehomogenising with respect to $z_0$, which vanishes at no point of $\widehat X$, identifies $(R_X)_j$ with $\mathcal E_j(X)$, so
\[
 H_X(j)=\dim_\C\mathcal E_j(X)\qquad(j\ge0),
\]
a non-decreasing function with $H_X(j)=N$ for large $j$.  Since $R_X$ is reduced and $z_0$ vanishes at no point of $\widehat X$, $z_0$ is a nonzerodivisor on $R_X$, and the \emph{Artinian reduction}
\[
 A:=R_X/(z_0)
\]
is a graded Artinian algebra with Hilbert function (the \emph{$h$-vector} of $\widehat X$)
\[
 \operatorname{Hilb}(A)_j=H_X(j)-H_X(j-1),\qquad\sum_j\operatorname{Hilb}(A)_j=N .
\]
The \emph{socle} of $A$ is $\Soc(A)=\{a\in A:a\cdot A_+=0\}$, where $A_+$ is the irrelevant ideal; since $A$ is generated in degree one, $\Soc(A)_j=\{a\in A_j:aA_1=0\}$.  The algebra $A$ is \emph{level} if its socle is concentrated in the top nonzero degree.
\end{definition}

\begin{definition}[Cayley--Bacharach property]\label{def:cb}
A finite set $\Gamma\subset\PP^n_\C$ is \emph{Cayley--Bacharach in degree $k$} if no form of degree $k$ vanishes at all points of $\Gamma$ except exactly one; equivalently, every form of degree $k$ vanishing at $|\Gamma|-1$ points of $\Gamma$ vanishes on all of $\Gamma$.
\end{definition}

For zero-dimensional complete intersections in the plane we use the classical Cayley--Bacharach theorem in the form of \cite{EGH96}: \emph{if two plane curves of degrees $a$ and $b$ meet in $ab$ distinct points, then every curve of degree $a+b-3$ passing through all but one of these points passes through all of them.}  Section~\ref{sec:generic} applies this with $(a,b)=(3,4)$, and handles the non-reduced case separately.

\subsection{Schoenberg coefficients}\label{sec:def-schoenberg}

For a finite $X\subset\Sph{d-1}$ and an orthonormal basis $(Y_{\ell r})_{r=1}^{h_\ell}$ of $\Harm_\ell$, the \emph{Schoenberg coefficients} of $X$ are
\[
 A_\ell(X)=\frac1{N^2}\sum_{x,y\in X}Q_\ell^{(d)}(x\cdot y)
 =\frac1{h_\ell N^2}\sum_{r=1}^{h_\ell}\Bigl|\sum_{x\in X}Y_{\ell r}(x)\Bigr|^2\ge0 ,
\]
the second expression being the addition formula.  Thus $X$ is a $t$-design if and only if $A_\ell(X)=0$ for $1\le\ell\le t$ \cite{DGS}.  Section~\ref{sec:residue} uses the \emph{signed} Schoenberg coefficients of a corank-one design, in which the summand carries the product $\eps_x\eps_y$ of the corank-one signs; they are nonnegative for the same reason.

\subsection{Table of symbols}\label{sec:symbols}

Table~\ref{tab:symbols} lists the recurring symbols and where each is defined.

\begin{table}[htbp]
\caption{Table of symbols.}\label{tab:symbols}
\footnotesize
\renewcommand{\arraystretch}{1.2}
\begin{tabular}{@{}>{\raggedright\arraybackslash}p{0.34\textwidth}>{\raggedright\arraybackslash}p{0.62\textwidth}@{}}
\toprule
Symbol & Meaning \\
\midrule
$X$, $N=|X|$, $x\cdot y$ & node set, its cardinality, Euclidean inner product \\
$\Pol_j$, $\Harm_j$, $h_j$, $D_m$ & polynomials of degree $\le j$ on the sphere; harmonics; $\dim\Harm_j$; $\dim\Pol_m$ \\
$\langle\cdot,\cdot\rangle_X$, $\langle\cdot,\cdot\rangle_{L^2}$ & sampled and spherical inner products \\
$\mathcal E_j(X)$, $\Harm_j^\C(X)$ & degree-$j$ evaluation space; evaluations of $\Harm_j^\C$ \\
$Q_j^{(d)}$, $P_\ell$ & zonal Gegenbauer polynomial ($Q_j^{(d)}(1)=1$); $P_\ell=Q_\ell^{(d)}$ in Section~\ref{sec:residue} \\
$K_m^{(d)}$, $K$ & reproducing kernel of $\Pol_m$; abbreviation when $(m,d)$ is fixed \\
$\sigma_d$, $w_d$, $\mu_j$ & inner-product distribution, its density and moments \\
$c=N-D_m$ & corank of a $2m$-design \\
$E$, $P=EE^{T}$, $Q=I-P$ & normalised evaluation matrix; projectors onto $\mathcal E_m(X)$ and $\mathcal E_m(X)^\perp$ \\
$u_x\in\Sph{c-1}$ & Gale vectors (Section~\ref{sec:def-frames}, Theorem~\ref{thm:finite-defect}) \\
$\eps_x\in\{\pm1\}$ & corank-one signs (Theorem~\ref{thm:split}) \\
$\zeta_x\in\Sph{1}\subset\C$ & corank-two phases (Section~\ref{sec:alias}) \\
$A_{\ell,1}$ & signed Schoenberg coefficients (Section~\ref{sec:residue}) \\
$\widehat X$, $R_X$, $H_X$, $A=R_X/(z_0)$ & projective lift, coordinate ring, Hilbert function, Artinian reduction \\
$\Soc(A)$, $\mathscr S_X$ & socle; its harmonic-quadratic model (Proposition~\ref{prop:socle-interpolation}) \\
$A_j$, $\eta_j$, $A$, $H$ & alias linear forms and harmonic quadratics ($A=\zeta\eta$ on $X$); $A=(A_1,A_2)^T$, $H=(\eta_1,\eta_2)^T$ in Sections~\ref{sec:cubic}--\ref{sec:moment} \\
$R$, $\FF$, $\LL$ & harmonic cubic $\det(A,H)$; Hermitian quartic matrix; Pauli pencil \\
$\sigma_1,\sigma_2,\sigma_3$ & Pauli matrices \\
\bottomrule
\end{tabular}

\end{table}

\section{Finite-defect structure at every corank}\label{sec:finite-defect}

Let $m\ge1$, let $X\subset\Sph{d-1}$ be a spherical $2m$-design, put $N=|X|$ and $c=N-D_m\ge1$, and abbreviate $K=K_m^{(d)}$.  As in Section~\ref{sec:def-frames}, choose a real orthonormal basis $(\phi_\alpha)$ of $\Pol_m$ and form the normalised evaluation matrix
\[
 E_{x\alpha}=N^{-1/2}\phi_\alpha(x).
\]
Then $E^*E=I_{D_m}$ by design exactness, so $P=EE^*$ is the orthogonal projector onto the degree-$m$ evaluation space $\mathcal E_m(X)$ and
\[
 P_{xy}=\frac1N K(x\cdot y).
\]

\begin{theorem}[Gale dual of a design]\label{thm:finite-defect}
Let $Q=I-P$.
\begin{enumerate}
\item[(a)] $Q$ is an orthogonal projection of rank $c$, with
\[
 Q_{xx}=\frac cN,
 \qquad
 Q_{xy}=-\frac1N K(x\cdot y)\quad(x\ne y).
\]
Consequently
\begin{equation}\label{eq:pairwise-fisher}
 |K(x\cdot y)|\le c\qquad(x\ne y).
\end{equation}
Every principal $(c+1)$-minor of $cI-K_X^{\mathrm{off}}$ vanishes, where $K_X^{\mathrm{off}}$ has zero diagonal and off-diagonal entry $K(x\cdot y)$.  Conversely, an $N$-point subset of $\Sph{d-1}$ with $N=D_m+c$ is a spherical $2m$-design if and only if $cI-K_X^{\mathrm{off}}$ is positive semidefinite of rank $c$.
\item[(b)] The Gale vectors $u_x$ of Section~\ref{sec:def-frames} are unit vectors, $u_x\in\Sph{c-1}$; as a multiset, $U=(u_x)_{x\in X}$ is a spherical $2$-design in $\R^c$ (repetitions allowed), and
\begin{equation}\label{eq:gale-coupling}
 \sum_{x\in X}p(x)u_x=0\quad(p\in\Pol_m),
 \qquad
 u_x\cdot u_y=-\frac{K(x\cdot y)}c\quad(x\ne y).
\end{equation}
Moreover, equality $|K(x\cdot y)|=c$ holds exactly when $u_y=-\operatorname{sgn}\bigl(K(x\cdot y)\bigr)\,u_x$.
\item[(c)] If $X$ contains an antipodal pair, then
\begin{equation}\label{eq:antipodal-bound}
 N\ge2\binom{d+m-1}{m}.
\end{equation}
If equality holds, then $u_{-x}=(-1)^{m+1}u_x$ for every antipodal pair occurring in $X$.
\end{enumerate}
\end{theorem}

\begin{proof}
The identities in (a) follow from $Q=I-P$.  Positivity of a $2\times2$ principal minor gives $Q_{xy}^2\le Q_{xx}Q_{yy}$, which is \eqref{eq:pairwise-fisher}; the $(c+1)$-minor statement follows from $\rank Q=c$.  Conversely, let
\[
 M=\frac1N\bigl(K(x\cdot y)\bigr)_{x,y\in X}.
\]
The zonal expansion makes $M$ positive semidefinite and $\tr M=D_m$.  If $I-M=N^{-1}(cI-K_X^{\mathrm{off}})$ is positive semidefinite of rank $c$, then the eigenvalues of $M$ lie in $[0,1]$ and at least $D_m=N-c$ of them equal $1$.  Since their sum is $D_m$, the remaining eigenvalues vanish.  Thus $M$ is a rank-$D_m$ projection, equivalently $E^*E=I_{D_m}$.  Products of two polynomials of degree at most $m$ span the restrictions of all polynomials of degree at most $2m$, so this is exactly the spherical $2m$-design condition.

Factor $Q=VV^T$ with $V^TV=I_c$, let $v_x$ be the row indexed by $x$, and put $u_x=\sqrt{N/c}\,v_x$; these are the Gale vectors of Section~\ref{sec:def-frames}.  The constant diagonal gives $|u_x|=1$.  Since the columns of $V$ span $\mathcal E_m(X)^\perp$, one has $\sum_xp(x)u_x=0$ for $p\in\Pol_m$.  In particular $\sum_xu_x=0$, while
\[
 \sum_xu_xu_x^T=\frac Nc V^TV=\frac Nc I_c.
\]
These are precisely the first- and second-moment identities for a spherical $2$-design.  The off-diagonal formula follows from $Q_{xy}=v_x\cdot v_y$, and equality in \eqref{eq:pairwise-fisher} is the equality case of Cauchy--Schwarz for the unit vectors $u_x,u_y$: $u_x\cdot u_y=-K(x\cdot y)/c=\mp1$ forces $u_y=\mp u_x$.

Finally, $Q_j^{(d)}(-1)=(-1)^j$ and the harmonic decomposition of homogeneous polynomials gives
\[
 K(-1)=(-1)^m\left\{\binom{d+m-1}{m}-\binom{d+m-2}{m-1}\right\}.
\]
Since $D_m$ is the sum of the two binomial coefficients, \eqref{eq:pairwise-fisher} at an antipodal pair yields
\[
 N=D_m+c\ge D_m+|K(-1)|=2\binom{d+m-1}{m}.
\]
The equality statement follows from part (b).
\end{proof}

\begin{proposition}[Cayley--Bacharach and the Hilbert tail]\label{prop:cb-hilbert}
Let $\widehat X\subset\PP^d_\C$ be the projective lift of $X$, with Artinian reduction $A=R_X/(z_0)$ (Definition~\ref{def:lift}).  Then $\widehat X$ is Cayley--Bacharach in degree $m$ (Definition~\ref{def:cb}), and
\[
 \operatorname{Hilb}(A)=(1,d,h_2,\ldots,h_m,a_{m+1},a_{m+2},\ldots),
 \qquad
 \sum_{j>m}a_j=c.
\]
If $c\le2$, then the degree-$(m+1)$ evaluation map is onto and
\[
 \operatorname{Hilb}(A)=(1,d,h_2,\ldots,h_m,c).
\]
\end{proposition}

\begin{proof}
If a degree-$m$ form vanished at all lifted nodes but one, say $x$, its evaluation vector would be a nonzero multiple of the coordinate vector $e_x$ and would lie in $\ran P=\mathcal E_m(X)$.  Then $Pe_x=e_x$, forcing $P_{xx}=1$, contrary to $P_{xx}=D_m/N<1$.

Evaluation is injective through degree $m$ (Definition~\ref{def:exactness}), so $H_X(j)=D_j$ for $j\le m$, and $\operatorname{Hilb}(A)_j=H_X(j)-H_X(j-1)=D_j-D_{j-1}=h_j$ for $1\le j\le m$ (Definition~\ref{def:lift}).  The remaining entries sum to $N-D_m=c$.

Assume $c\le2$ and suppose $0\ne w\in\mathcal E_{m+1}(X)^\perp$.  Then $w\in W:=\mathcal E_m(X)^\perp$ and, for each coordinate $x_j$ and each $p\in\mathcal E_m(X)$,
\[
 \langle x_jw,p\rangle_X=\langle w,x_jp\rangle_X=0.
\]
Thus $w,x_1w,\ldots,x_dw$ lie in the $c$-dimensional space $W$.  If $c=1$, the constant positive diagonal of the projection onto $W$ makes a spanning vector of $W$ nowhere zero, so every coordinate is constant on $X$, impossible.  If $c=2$, multiplication by the nonzero coordinates of $w$ shows that the affine evaluation matrix on $\operatorname{supp}w$ has rank at most two.  The support therefore lies on a real affine line and contains at most two points of the sphere.  Orthogonality to constants and coordinates excludes supports of size one or two.  Hence $\mathcal E_{m+1}(X)=\C^X$.
\end{proof}

\begin{proposition}[Uniform multiplicative relations]\label{prop:relation-count}
Choose the Gale vectors in Theorem~\ref{thm:finite-defect} and put $\widehat x=(1,x)\in\R^{d+1}$.  The space
\[
 \mathscr R_X=\{B\in\R^{(d+1)\times c}:\widehat x^{T}Bu_x=0\ \text{for every }x\in X\}
\]
satisfies
\[
 \dim_{\R}\mathscr R_X\ge \max\{0,dc-h_m\}.
\]
\end{proposition}

\begin{proof}
Identify the $c$ Gale coordinates with an orthonormal basis of $W=\mathcal E_m(X)^\perp$.  Multiplication defines a map
\[
 \mathcal E_1(X)\otimes W\longrightarrow\C^X.
\]
Its image is orthogonal to $\mathcal E_{m-1}(X)$, since $w\perp\mathcal E_m(X)$ and an affine linear function times a polynomial of degree at most $m-1$ has degree at most $m$.  Thus the image has dimension at most
\[
 N-D_{m-1}=h_m+c.
\]
The domain has dimension $(d+1)c$, and its kernel is the complexification of $\mathscr R_X$: the Gale coordinate functions are the columns of $V$, which differ from the coordinates of $u_x$ by the common factor $\sqrt{c/N}$, and all data are real, so real and complex kernels have the same dimension.  Rank--nullity gives $(d+1)c-(h_m+c)=dc-h_m$.
\end{proof}

\begin{remark}
Theorem~\ref{thm:finite-defect}(c) is sharp: the icosahedron has $(d,m,N,c)=(3,2,12,3)$ and attains \eqref{eq:antipodal-bound}.  For $4$-designs, a single antipodal pair already forces $N\ge d(d+1)$.
\end{remark}

\section{Corank-one splitting at arbitrary even strength}\label{sec:split}

\begin{theorem}[Corank-one splitting]\label{thm:split}
Let $X\subset\Sph{d-1}$ be an equal-weight spherical $2m$-design with $|X|=D_m+1$. Then $|X|$ is even and there is a partition
\[
X=X_+\sqcup X_-,\qquad |X_+|=|X_-|=|X|/2,
\]
such that $X_+$ and $X_-$ are each spherical $m$-designs. If $K_m$ is the reproducing kernel of $\Pol_m$, there are signs $\eps_x\in\{\pm1\}$, constant on the two parts, such that for distinct nodes
\[
K_m(x\cdot y)=-\eps_x\eps_y .
\]
\end{theorem}

\begin{proof}
This is Theorem~\ref{thm:finite-defect}(b) with $c=1$; we give the short direct argument.  Choose an orthonormal basis $\phi_1,\dots,\phi_{D_m}$ of $\Pol_m$ and form the normalised evaluation matrix $E_{x\alpha}=|X|^{-1/2}\phi_\alpha(x)$. Design exactness gives $E^*E=I_{D_m}$. Therefore $P=EE^*$ is an orthogonal projection of rank $D_m=|X|-1$. The reproducing kernel has diagonal $D_m$, so
\[
P_{xx}=\frac{D_m}{|X|},\qquad (I-P)_{xx}=\frac1{|X|}.
\]
The rank-one projection $I-P$ consequently has the form $I-P=|X|^{-1}\eps\eps^T$ for a sign vector $\eps\in\{\pm1\}^X$. The constant evaluation vector lies in $\ran P$, so $\sum_x\eps_x=0$; thus the signs split equally. Moreover $\eps\perp\ran E$, and hence
\begin{equation}\label{eq:signed}
\sum_{x\in X}\eps_xp(x)=0\qquad(p\in\Pol_m).
\end{equation}
Combining this signed identity with the ordinary design identities gives exactness on each sign class separately through degree $m$. Finally, the off-diagonal entries of $I-P$ give the kernel identity.
\end{proof}

\begin{remark}
The theorem is an exact finite-defect statement: one point of excess over the degree-$m$ evaluation dimension is not an arbitrary perturbation of tightness but forces a signed decomposition into two lower-strength designs. Equivalently, the vectors $\eps_x(\phi_\alpha(x))_\alpha\in\R^{D_m}$ form a regular simplex.
\end{remark}

Throughout Sections~\ref{sec:split}--\ref{sec:higher} we write $2n:=|X|=D_m+1$ and $K:=K_m^{(d)}$. The splitting has a quantitative consequence: the distance distribution from any point is forced.

\begin{lemma}[Forced distance distribution]\label{lem:distribution}
Let $X$ be as in Theorem~\ref{thm:split}, $m\ge1$, and fix $x\in X$. Let $\mathcal A$ be the set of distinct complex roots of $K+1$ and $\mathcal B$ the set of distinct complex roots of $K-1$, and put
\[
k_\alpha=\#\{y\in X_{\eps_x}\setminus\{x\}:\ x\cdot y=\alpha\},\qquad
l_\beta=\#\{y\in X_{-\eps_x}:\ x\cdot y=\beta\},
\]
which vanish for non-real roots and for roots outside $[-1,1)$. Then for every polynomial $p$ of degree at most $2m$
\begin{equation}\label{eq:quadrature}
p(1)+\sum_{\alpha\in\mathcal A}k_\alpha p(\alpha)+\sum_{\beta\in\mathcal B}l_\beta p(\beta)=2n\int_{-1}^1p\,w_d\dd t .
\end{equation}
The vector $(1,(k_\alpha),(l_\beta))$ is the unique solution of \eqref{eq:quadrature}; in particular the numbers $k_\alpha,l_\beta$ do not depend on $x$, and
\begin{equation}\label{eq:sums}
\sum_\alpha k_\alpha=n-1,\qquad \sum_\beta l_\beta=n,\qquad \sum_\alpha k_\alpha\alpha=-1,\qquad \sum_\beta l_\beta\beta=0 .
\end{equation}
\end{lemma}

\begin{proof}
For $y\in X_{\eps_x}\setminus\{x\}$ we have $K(x\cdot y)=-1$, so $x\cdot y\in\mathcal A$; for $y\in X_{-\eps_x}$ we have $K(x\cdot y)=1$, so $x\cdot y\in\mathcal B$. Summing $p(x\cdot y)$ over $y\in X$ and using that $X$ is a $2m$-design gives \eqref{eq:quadrature}. The nodes $1$, $\mathcal A$, $\mathcal B$ are pairwise distinct, since $K$ takes the values $D_m>1$, $-1$, $+1$ on them, and there are at most $2m+1$ of them, whereas \eqref{eq:quadrature} for $p=1,t,\dots,t^{2m}$ is a Vandermonde system of $2m+1$ equations; hence the solution is unique. The first two identities in \eqref{eq:sums} are the definitions. For the last two, $p(y)=x\cdot y$ lies in $\Pol_1\subset\Pol_m$, so \eqref{eq:signed} gives $\sum_{y\in X}\eps_y\,x\cdot y=0$; together with $\sum_{y\in X}x\cdot y=0$ this yields $\sum_{y\in X_{\eps_x}}x\cdot y=0=\sum_{y\in X_{-\eps_x}}x\cdot y$, i.e.\ $1+\sum_\alpha k_\alpha\alpha=0$ and $\sum_\beta l_\beta\beta=0$.
\end{proof}

\begin{lemma}[Galois constancy]\label{lem:galois}
In the situation of Lemma~\ref{lem:distribution}, $k_{\gamma(\alpha)}=k_\alpha$ and $l_{\gamma(\beta)}=l_\beta$ for every $\gamma\in\Gal(\overline\Q/\Q)$. Consequently $k$ (resp.\ $l$) is constant on the root set of each irreducible factor of $K+1$ (resp.\ $K-1$) over $\Q$, and an irreducible factor having a non-real root, or a root outside $[-1,1)$, has weight zero on all of its roots.
\end{lemma}

\begin{proof}
The system \eqref{eq:quadrature} has rational data (the moments $\mu_j$, the mass $2n$ and the node $1$), and $\gamma$ permutes $\mathcal A$ and $\mathcal B$. Applying $\gamma$ to a solution gives a solution of the same system with the nodes permuted; by uniqueness, $w(\gamma(\tau))=\gamma(w(\tau))$ for every node $\tau$. The weights are integers, so $\gamma(w(\tau))=w(\tau)$.
\end{proof}

\section{All-dimensional classification at strength four}\label{sec:strength4}

\begin{theorem}[Corank-one classification at strength four]\label{thm:strength4}
Let $d\ge2$. Suppose $X\subset\Sph{d-1}$ is an equal-weight spherical $4$-design with
\[
|X|=D_2+1=\frac{(d+1)(d+2)}{2}.
\]
Then $d=2$, and $X$ is a regular hexagon.
\end{theorem}

\begin{proof}
By Theorem~\ref{thm:split}, the design splits into two spherical $2$-designs of equal size
\[
n=\frac{|X|}2=\frac{(d+1)(d+2)}4 .
\]
For two distinct points in the same part, \eqref{eq:K2} gives $K_2^{(d)}(t)=-1$. The two possible inner products are
\begin{equation}\label{eq:alphabeta}
\alpha,\beta=\frac{-1\pm\tau}{d+2},\qquad \tau^2=\frac{d^2+d-4}{d}.
\end{equation}
Fix one part $Y$, and let $\mathsf A$ be the adjacency matrix for inner product $\alpha$. Its Gram matrix is
\[
\mathsf G=(1-\beta)I+(\alpha-\beta)\mathsf A+\beta J .
\]
Since $Y$ is a spherical $2$-design,
\[
\mathsf G\mathbf 1=0,\qquad \mathsf G^2=\frac nd\,\mathsf G .
\]
It follows from $\mathsf G\mathbf1=0$ that $\mathsf A$ is regular, with valency
\begin{equation}\label{eq:valency}
k=\frac{n-1}2+\frac{n-d-3}{2\tau}.
\end{equation}
If $\tau$ were irrational, integrality of $k$ would force $n-d-3=0$, equivalently $d^2-d-10=0$, which has no integer solution. Hence $\tau\in\Q$.

Assume $d>2$. Since $\mathsf G$ is the Gram matrix of $Y$, and $Y$ spans $\R^d$ by the second-moment identity, $\rank\mathsf G=d$; hence the subspace $\mathbf1^\perp\cap\ker\mathsf G$ has dimension
\[
n-d-1=\frac{(d-2)(d+1)}4>0 .
\]
On this subspace, $\mathsf A$ has eigenvalue
\[
r=-\frac{1-\beta}{\alpha-\beta}=-\frac12-\frac{d+3}{2\tau}.
\]
This eigenvalue is rational. Since $\mathsf A$ is an integer matrix, it is an algebraic integer and hence an integer. Thus
\[
\nu:=\frac{d+3}{\tau}=-2r-1
\]
is a positive odd integer. Squaring and using \eqref{eq:alphabeta},
\begin{equation}\label{eq:nu2}
\nu^2=\frac{d(d+3)^2}{d^2+d-4}.
\end{equation}
For $d\ge10$,
\[
\nu^2-(d+5)=\frac{4(2d+5)}{d^2+d-4}>0
\qquad\text{and}\qquad
(d+6)-\nu^2=\frac{d^2-7d-24}{d^2+d-4}>0 .
\]
Thus the integer $\nu^2$ lies strictly between the consecutive integers $d+5$ and $d+6$, which is impossible.

It remains to consider $2\le d\le9$. The equal sign split requires $|X|$ even, leaving only $d=2,3,6,7$. For $d=3,6,7$ one has
\[
\tau^2=\frac83,\quad\frac{19}3,\quad\frac{52}7,
\]
respectively, contradicting $\tau\in\Q$. Hence $d=2$.

Each sign class is then a three-point spherical $2$-design on $\Sph{1}$, so each is a regular triangle. Write the two triangles as the roots of $z^3=a$ and $z^3=b$ on the unit circle. The degree-three moment of the full $4$-design vanishes, giving $a+b=0$. Hence the union is the regular hexagon.
\end{proof}

\begin{corollary}\label{cor:ten}
There is no ten-point equal-weight spherical $4$-design on $\Sph{2}$.
\end{corollary}

\begin{remark}
For $m=2$ the polynomials $K_2^{(d)}\pm1$ are quadratics, irreducible over $\Q$ unless $d(d^2+d-4)$, respectively $d(d+1)(d+4)$, is a perfect square. Theorem~\ref{thm:galois} therefore gives a second proof of Theorem~\ref{thm:strength4} whenever at least one of these two quantities is not a square. Only dimensions represented simultaneously by integral points on both elliptic curves escape that argument; the spectral proof above avoids this Diophantine problem altogether.
\end{remark}

\section{Arithmetic obstructions at corank one, and strength six}\label{sec:higher}

We now exploit Lemmas~\ref{lem:distribution} and \ref{lem:galois} for arbitrary $m$. Two classical facts about the kernel are needed.

\begin{fact}\label{fact:jacobi}
$K_m^{(d)}(t)$ is a constant multiple of the Jacobi polynomial $P_m^{(\frac{d-1}2,\frac{d-3}2)}(t)$ \textup{\cite{DGS,Levenshtein98}}. The zeros of $P_n^{(a,b)}$ sum to $n(b-a)/(2n+a+b)$: in the explicit representation \textup{\cite[Chapter~IV]{Szego}}
\[
 P_n^{(a,b)}(t)=\frac{\Gamma(a+n+1)}{n!\,\Gamma(a+b+n+1)}\sum_{k=0}^n\binom nk\frac{\Gamma(a+b+n+k+1)}{\Gamma(a+k+1)}\Bigl(\frac{t-1}2\Bigr)^k,
\]
only the terms $k=n$ and $k=n-1$ contribute to the coefficients of $t^n$ and $t^{n-1}$, whose ratio is $-n+2n(a+n)/(2n+a+b)$. Hence the zeros of $K_m^{(d)}$ sum to
\[
s_{d,m}:=-\frac{m}{2m+d-2}\ne0 .
\]
For $m\ge2$ the polynomials $K\pm1$ have the same $t^m$ and $t^{m-1}$ coefficients as $K$, so their roots also sum to $s_{d,m}$.
\end{fact}

\begin{lemma}[Integral form of the kernel]\label{lem:integral}
$m!\,K_m^{(d)}(t)\in\Z[d,t]$. The coefficient of $t^m$ in this polynomial is $\prod_{i=0}^{m-1}(d+2i)$, and the polynomial is divisible in $\Q[d,t]$ by
\[
g_m(d):=\prod_{i=0}^{\lfloor m/2\rfloor-1}(d+2i).
\]
\end{lemma}

\begin{proof}
Write $\lambda=d/2-1$, so $Q_k^{(d)}=C_k^{(\lambda)}/C_k^{(\lambda)}(1)$ with $C_k^{(\lambda)}$ the classical Gegenbauer polynomial. From $h_k=(2k+d-2)(d-1)_{k-1}/k!$ and $C_k^{(\lambda)}(1)=(d-2)_k/k!$ (Pochhammer symbols) one gets $h_kQ_k^{(d)}=(2k+d-2)C_k^{(\lambda)}(t)/(d-2)$, and inserting the expansion $C_k^{(\lambda)}(t)=\sum_j(-1)^j\frac{(\lambda)_{k-j}}{j!(k-2j)!}(2t)^{k-2j}$ with $(\lambda)_{k-j}=2^{-(k-j)}\prod_{i=0}^{k-j-1}(d-2+2i)$, the coefficient of $t^{k-2j}$ in $h_kQ_k^{(d)}$ equals
\[
(-1)^j2^{-j}\,(2k+d-2)\prod_{i=1}^{k-j-1}(d-2+2i)\Big/\bigl(j!\,(k-2j)!\bigr),
\]
a polynomial in $d$ whose denominator divides $2^jj!(k-2j)!$, which divides $k!$, which divides $m!$. This proves integrality. The leading coefficient is the case $k=m$, $j=0$: $(2m+d-2)\prod_{i=1}^{m-1}(d-2+2i)/m!=\prod_{i=0}^{m-1}(d+2i)/m!$. For the divisibility, use Fact~\ref{fact:jacobi} in hypergeometric form,
\begin{align*}
K_m^{(d)}(t)
 &=D_m\cdot{}_2F_1\!\left(-m,\,m+d-1;\,\frac{d+1}{2};\,\frac{1-t}{2}\right),\\
D_m&=\frac{(d+2m-1)(d)_{m-1}}{m!}.
\end{align*}
where $(d)_{m-1}=d(d+1)\cdots(d+m-2)$. The ${}_2F_1$ is a finite sum whose coefficients have denominators $\prod_{i<s}(d+1+2i)$, which do not vanish at even negative $d$. Hence the polynomial $m!K_m^{(d)}(t)$ vanishes identically in $t$ at $d=-2i$ whenever $0\le 2i\le m-2$, i.e.\ for $0\le i\le\lfloor m/2\rfloor-1$, so each $d+2i$ divides it.
\end{proof}

\begin{theorem}[Galois obstruction]\label{thm:galois}
Let $m\ge2$ and $d\ge3$. If $K_m^{(d)}-1$ or $K_m^{(d)}+1$ is irreducible over $\Q$, there is no corank-one spherical $2m$-design on $\Sph{d-1}$.
\end{theorem}

\begin{proof}
Suppose a corank-one design exists and $K-1$ is irreducible. By Lemma~\ref{lem:galois} all $l_\beta$ are equal, say to $l$. If some root of $K-1$ is non-real or lies outside $[-1,1)$, then $l=0$, contradicting $\sum_\beta l_\beta=n\ge1$. Otherwise $\sum_\beta l_\beta=ml=n$ and, by Fact~\ref{fact:jacobi}, $\sum_\beta l_\beta\beta=l\,s_{d,m}\ne0$, contradicting \eqref{eq:sums}.

Suppose $K+1$ is irreducible. By Lemma~\ref{lem:galois}, all $k_\alpha$ are equal, say to $k$. If some root of $K+1$ is non-real or lies outside $[-1,1)$, then $k=0$, contradicting $\sum_\alpha k_\alpha=n-1\ge1$. Otherwise $k=(n-1)/m$, and \eqref{eq:sums} gives $k\,s_{d,m}=-1$, i.e.\ $n-1=2m+d-2$, i.e.\ $D_m=2n-1=4m+2d-3$. This never happens for $d\ge3$, $m\ge2$: put $\Delta(d,m)=D_m-4m-2d+3$. Then $\Delta(d,m+1)-\Delta(d,m)=h_{m+1}-4\ge 2(m+1)+1-4>0$, because $h_k=\dim\Harm_k(\R^d)$ is non-decreasing in $d$ and $h_k^{(3)}=2k+1$; further $\Delta(d,2)=(d^2-d-10)/2$ is positive for $d\ge4$ and never zero, while $\Delta(3,2)=-2$ and $\Delta(3,3)=1$.
\end{proof}

\begin{remark}
The hypothesis $m\ge2$ is essential: for $m=1$ the constant term is the $t^{m-1}$ coefficient, $K_1-1=dt$ has the root $0$, and the double-simplex examples exist in every even dimension \cite{Mimura90}.
\end{remark}

\begin{theorem}[An unconditional family in every strength]\label{thm:eisenstein}
Let $m\ge2$, $d\ge3$, and let $p>m$ be a prime dividing $d+2i$ for some $0\le i\le\lfloor m/2\rfloor-1$, with $v:=\vp(d+2i)$ satisfying $\gcd(v,m)=1$. Then $K_m^{(d)}+1$ and $K_m^{(d)}-1$ are both irreducible over $\Q$, and hence there is no corank-one spherical $2m$-design on $\Sph{d-1}$.
\end{theorem}

\begin{proof}
Let $f=m!(K\pm1)=m!K(d,t)\pm m!\in\Z[t]$. By Lemma~\ref{lem:integral}, every coefficient of $m!K(d,t)$ is divisible by $(d+2i)$, hence by $p^v$. The constant term of $f$ is congruent to $\pm m!\not\equiv0\pmod p$ because $p>m$. The leading coefficient is $\prod_{i'<m}(d+2i')$, and $p\mid d+2i'$ would force $p\mid 2(i'-i)$ with $|i'-i|<m<p$, so $\vp(\operatorname{lead}f)=v$. The reversed polynomial $f^*(t)=t^mf(1/t)$ therefore has a $p$-adic Newton polygon consisting of the single segment from $(0,v)$ to $(m,0)$: every intermediate coefficient has valuation at least $v$, above the segment. By the Newton polygon theorem (Dumas \cite{Dumas06}; see \cite[Ch.~2]{Prasolov}), every irreducible factor of $f^*$ over $\Q_p$ has degree divisible by $m/\gcd(v,m)=m$, so $f^*$, and with it $f$ (note $f(0)\ne0$), is irreducible over $\Q$. Theorem~\ref{thm:galois} concludes.
\end{proof}

\begin{example}
For $m=2$: every $d$ having a prime $p\ge3$ to an odd power. For $m=3$: every $d$ having a prime $p\ge5$ with $\gcd(\vp(d),3)=1$. For $m=4$: every $d$ such that $d$ or $d+2$ has a prime $p\ge5$ to an odd power. In every strength the excluded dimensions have positive density.
\end{example}

The strength-six classification needs the rational roots of $K_3^{(d)}-1$.

\begin{lemma}[Rational roots at strength six]\label{lem:rational-roots}
For an integer $d\ge2$ put $g_d(t)=d(d+4)t^3+3dt^2-3dt-3\in\Z[t]$.  Then
\begin{align*}
K_3^{(d)}(t)&=\frac d6\Bigl((d+2)(d+4)t^3+3(d+2)t^2-3(d+2)t-3\Bigr),\\
K_3^{(d)}(t)-1&=\frac{d+2}{6}\,g_d(t).
\end{align*}
Every real root of $g_d$ lies in $(-1,1)$, and the only pairs $(d,t)$ with $t\in\Q$ a root of $g_d$ are
\[
(d,t)\in\{(2,-\tfrac12),\ (6,\tfrac12),\ (12,-\tfrac12)\}.
\]
\end{lemma}

\begin{proof}
The formula for $K_3^{(d)}$ is a direct computation from \eqref{eq:kernel}.  The location of the real roots is elementary, and the determination of the rational roots is a finite Diophantine certificate (a discriminant squeeze together with an explicit enumeration of the remaining cases); both are carried out in Appendix~\ref{app:rational-roots}.
\end{proof}

\begin{theorem}[Corank-one classification at strength six]\label{thm:strength6}
Let $X\subset\Sph{d-1}$ be an equal-weight spherical $6$-design with $|X|=D_3+1=d(d+1)(d+5)/6+1$. Then $d=2$ and $X$ is a regular octagon.
\end{theorem}

\begin{proof}
(i) \emph{Parity.} For $d=2e$, $D_3=e(2e+1)(2e+5)/3$ with $(2e+1)(2e+5)$ odd, so $v_2(D_3)=v_2(e)$ and $D_3$ is odd iff $e$ is odd. For $d$ odd, $d+1$ and $d+5$ are both even, so $4\mid(d+1)(d+5)$ and $D_3$ is even. Hence $|X|=D_3+1$ is even iff $d\equiv2\pmod4$, and Theorem~\ref{thm:split} forces $d\equiv2\pmod 4$.

(ii) \emph{$K-1$ has a rational root.} Otherwise the cubic $K_3^{(d)}-1$ is irreducible over $\Q$ and Theorem~\ref{thm:galois} applies ($d\ge3$).

(iii) \emph{Rational roots of $K-1$.} By Lemma~\ref{lem:rational-roots}, the only integers $d\ge2$ for which $K_3^{(d)}-1$ has a rational root are $d=2,6,12$.

(iv) $d=12$ gives $|X|=443$, odd, excluded by (i). For $d=6$,
\[
K_3^{(6)}(t)+1=2\bigl(40t^3+12t^2-12t-1\bigr),
\]
and the cubic $40t^3+12t^2-12t-1$ has no rational root: the candidates
\[
\pm1,\ \pm\tfrac12,\ \pm\tfrac14,\ \pm\tfrac15,\ \pm\tfrac18,\ \pm\tfrac1{10},\ \pm\tfrac1{20},\ \pm\tfrac1{40}
\]
all fail. Hence it is irreducible, and Theorem~\ref{thm:galois} applies. Directly: all $k_\alpha$ would equal $(n-1)/3=38/3$.

(v) $d=2$ is Proposition~\ref{prop:circle}.
\end{proof}

\begin{proposition}[The circle]\label{prop:circle}
For $m\ge1$, an equal-weight $2m$-design on $\Sph{1}$ with $2m+2$ points is a regular $(2m+2)$-gon; conversely the regular $(2m+2)$-gon is a $(2m+1)$-design.
\end{proposition}

This proposition is also a special case of Hong's classification of small circle designs \cite{Hong82}. We include the following proof to exhibit how it follows from the finite-defect splitting.
\begin{proof}
$D_m=2m+1$ on $\Sph{1}$, so by Theorem~\ref{thm:split} $X$ splits into two $(m+1)$-point $m$-designs. Writing points as unit complex numbers $z_j$, an $(m+1)$-point $m$-design has vanishing power sums $p_1=\dots=p_m=0$, so by Newton's identities $e_1=\dots=e_m=0$ and $\prod_j(z-z_j)=z^{m+1}-a$: a regular $(m+1)$-gon. The union of $\{z^{m+1}=a\}$ and $\{z^{m+1}=b\}$ has $(m+1)$-st power sum $(m+1)(a+b)$, which must vanish since $m+1\le2m$; so $b=-a$ and $X=\{z^{2m+2}=a^2\}$. The converse is classical.
\end{proof}

\begin{remark}[What remains]\label{rem:what-remains}
For each fixed $m$ the remaining question is for which $d$ both $K_m^{(d)}\pm1$ are reducible over $\Q$. For $m=2$ this is an integer-point problem on two elliptic curves. For $m=3$ the cubic case is complete: Lemma~\ref{lem:rational-roots} leaves $d=2,6,12$ for $K_3^{(d)}-1$; at $d=6$ the companion polynomial is proved irreducible in the proof of Theorem~\ref{thm:strength6}, while
\[
 K_3^{(12)}(t)+1=448t^3+84t^2-84t-5
\]
is irreducible by the rational-root test (and $d=12$ is independently excluded by parity). For $m\ge4$ reducibility may occur without rational roots, and we do not pursue a uniform arithmetic classification; exact factorisations over a large range of parameters (Appendix~\ref{app:factorisation-ranges}) show that at least one of $K_m^{(d)}\pm1$ is irreducible in every parity-eligible case tested. The eigenvalue argument of Section~\ref{sec:strength4} does not extend: the parts are $m$-distance $m$-designs, and the Delsarte--Goethals--Seidel threshold $t\ge2s-2$ for an association scheme holds only for $m\le2$.
\end{remark}

\section{Corank one at every strength: a residue formula}\label{sec:residue}

Let $X\subset\Sph{d-1}$ be a spherical $2m$-design with $N=D_m+1$, where $m\ge2$ and $d\ge2$.  (The hypothesis $d\ge3$ enters only in Theorem~\ref{thm:sign-obstruction}; Remark~\ref{rem:polygon-check} uses $d=2$.)  Let $\eps:X\to\{\pm1\}$ be the sign function of Theorem~\ref{thm:split}, so
\[
 K(x\cdot y)=-\eps_x\eps_y\qquad(x\ne y),
\]
where $K=K_m^{(d)}$.  Write $P_\ell=Q_\ell^{(d)}$ and $\sigma=\sigma_d$.  The signed Schoenberg coefficients are
\begin{equation}\label{eq:signed-schoenberg}
 A_{\ell,1}:=\frac1{N^2}\sum_{x,y\in X}\eps_x\eps_yP_\ell(x\cdot y)
 =\frac1{h_\ell N^2}\sum_{r=1}^{h_\ell}\left|\sum_{x\in X}\eps_xY_{\ell r}(x)\right|^2\ge0,
\end{equation}
where $(Y_{\ell r})$ is an orthonormal basis of $\Harm_\ell$; the second equality is the addition formula, exactly as for the unsigned Schoenberg coefficients of Section~\ref{sec:def-schoenberg}.

Let
\[
 \nu_X=\frac1{N^2}\sum_{x,y\in X}\delta_{x\cdot y},
 \qquad
 \Omega(t)=(t-1)(K(t)^2-1),
\]
and, for a polynomial $F$, let $[t^{-1}]F$ denote the coefficient of $t^{-1}$ in its Laurent expansion at infinity.

\begin{lemma}[Remainder functional]\label{lem:remainder-functional}
For polynomials $g$ and $\Omega$ with $\deg\Omega\ge1$, define
\[
 \widehat\Omega(t)=\int_{-1}^1\frac{\Omega(t)-\Omega(u)}{t-u}\dd\sigma(u).
\]
Then
\[
 \int_{-1}^1(g\bmod\Omega)(u)\dd\sigma(u)
 =[t^{-1}]\frac{\widehat\Omega(t)g(t)}{\Omega(t)}.
\]
\end{lemma}

\begin{proof}
Let $S(t)=\int(t-u)^{-1}\dd\sigma(u)=\sum_{j\ge0}\mu_jt^{-j-1}$.  If $n=\deg\Omega$, then
\[
 \frac{\widehat\Omega(t)}{\Omega(t)}=S(t)+O(t^{-n-1}).
\]
Both sides of the claimed identity are linear in $g$.  They vanish on multiples of $\Omega$, while for $\deg g<n$ the error term is $O(t^{-2})$ and $[t^{-1}](gS)=\int g\dd\sigma$.
\end{proof}

Define
\begin{equation}\label{eq:G-def}
 G(t)=\int_{-1}^1 (u-1)\frac{K(u)-K(t)}{u-t}\dd\sigma(u).
\end{equation}
This is a polynomial of degree at most $m-1$, with $G(1)=1-D_m$ and leading coefficient $-\operatorname{lead}(K)$.

\begin{lemma}\label{lem:omega-identity}
For every polynomial $g$,
\begin{equation}\label{eq:measure-remainder}
 \int g\dd\nu_X=\int(g\bmod\Omega)\dd\sigma
 =[t^{-1}]\frac{\widehat\Omega(t)g(t)}{\Omega(t)},
\end{equation}
and
\begin{equation}\label{eq:omega-hat}
 \widehat\Omega=GK+K^2-1.
\end{equation}
\end{lemma}

\begin{proof}
The diagonal of $\nu_X$ is supported at $t=1$, and every off-diagonal inner product is a zero of $K^2-1$.  Hence $\nu_X$ is supported on the zeros of $\Omega$.  The remainder modulo $\Omega$ has degree at most $2m$, so it has the same integral against $\nu_X$ and $\sigma$ by the design property.  Lemma~\ref{lem:remainder-functional} gives \eqref{eq:measure-remainder}.

Put
\[
 I(t)=\int\frac{K(t)-K(u)}{t-u}\dd\sigma(u).
\]
The reproducing identity \eqref{eq:reproduce-one} gives
\[
 \widehat\Omega(t)=(t-1)K(t)I(t)+K(t)-1.
\]
On the other hand, \eqref{eq:G-def} and $u/(t-u)=t/(t-u)-1$ give
\[
 G(t)=(t-1)I(t)-K(t)+1.
\]
Substitution proves \eqref{eq:omega-hat}.
\end{proof}

\begin{theorem}[Residue formula for the signed distribution]\label{thm:residue}
For every $\ell\ge0$,
\begin{equation}\label{eq:residue-formula}
 A_{\ell,1}=-[t^{-1}]\frac{G(t)P_\ell(t)}{(t-1)(K(t)^2-1)}.
\end{equation}
In particular $A_{\ell,1}=0$ for $\ell\le m$.  Put
\[
 q_m(d)=d^2+dm-2d-2m^2-2m.
\]
Then
\begin{align}
 A_{m+1,1}
 &=\binom{d+m-2}{m}^{-1},\label{eq:A-m1}\\
 A_{m+2,1}
 &=\frac{(d-2)(d+2m)}{d(d+2m-2)}\binom{d+m-2}{m}^{-1},\label{eq:A-m2}\\
 A_{m+3,1}
 &=\frac{(d-2)(d+2m)(d+m-1)q_m(d)}{d(d+2m-2)^2}
   \frac{(d-2)!m!}{(d+m)!}.\label{eq:A-m3}
\end{align}
\end{theorem}

\begin{proof}
Off the diagonal $\eps_x\eps_y=-K(x\cdot y)$, while $K(1)=D_m$ and $N=D_m+1$.  Therefore
\[
 A_{\ell,1}=1-\int KP_\ell\dd\nu_X.
\]
Using Lemma~\ref{lem:omega-identity}, then $K^2=(K^2-1)+1$, gives
\begin{align*}
 \int KP_\ell\dd\nu_X
 &=[t^{-1}]\frac{(GK+K^2-1)KP_\ell}{(t-1)(K^2-1)}\\
 &=G(1)+K(1)+[t^{-1}]\frac{GP_\ell}{(t-1)(K^2-1)}\\
 &=1+[t^{-1}]\frac{GP_\ell}{(t-1)(K^2-1)}.
\end{align*}
This proves \eqref{eq:residue-formula}.  Since $\deg G\le m-1$, the rational function there is $O(t^{\ell-m-2})$, proving the vanishing for $\ell\le m$.

For $\ell=m+j$, $1\le j\le3$, expand at infinity with $s=t^{-1}$:
\begin{align*}
 K&=k_mt^m(1+\kappa_1s+\kappa_2s^2+O(s^3)),\\
 G&=-k_mt^{m-1}(1+\gamma_1s+\gamma_2s^2+O(s^3)),\\
 P_{m+j}&=p_{m+j}t^{m+j}(1+O(s^2)).
\end{align*}
The exact coefficients are
\begin{gather*}
 \kappa_1=\frac{m}{d+2m-2},\qquad
 \kappa_2=-\frac{m(m-1)}{2(d+2m-2)},\\
 \gamma_1=\frac{(d-2)(m-1)}{d(d+2m-2)},\qquad
 \gamma_2=-\frac{(m-2)(dm+d-2)}{2d(d+2m-2)},
\end{gather*}
and, for $j=3$, the relative $s^2$ coefficient of $P_{m+3}$ is
\[
 \pi_2=-\frac{(m+2)(m+3)}{2(2m+d+2)}.
\]
Also
\begin{equation}\label{eq:lead-ratio}
 \frac{p_{m+j}}{k_m}
 =\frac1{h_m}\prod_{i=0}^{j-1}\frac{2m+2i+d-2}{m+i+d-2}.
\end{equation}
The term $-1$ in $K^2-1$ first appears at relative order $s^{2m}$ and therefore does not affect these three coefficients.  Coefficient extraction now gives \eqref{eq:A-m1}--\eqref{eq:A-m3}; the complete rational simplification is displayed in Appendix~\ref{app:residue-coefficients}.
\end{proof}

\begin{remark}[Consistency check on the circle]\label{rem:polygon-check}
For $d=2$ the corank-one $2m$-designs are exactly the regular $(2m+2)$-gons (Proposition~\ref{prop:circle}), with alternating signs $\eps$ around the polygon.  There $Q_\ell^{(2)}=T_\ell$, and a direct computation with roots of unity gives $A_{\ell,1}=1$ if $\ell\equiv m+1\pmod{2m+2}$ and $A_{\ell,1}=0$ otherwise.  Formula \eqref{eq:residue-formula} is valid for $d=2$ and reproduces these values; in particular \eqref{eq:A-m1}--\eqref{eq:A-m3} give $1,0,0$, the last two because of the factor $d-2$.  The ancillary script of Appendix~\ref{app:repro} verifies this for $\ell\le4m+3$ and $m\le4$.
\end{remark}

\begin{theorem}[Uniform corank-one obstruction]\label{thm:sign-obstruction}
Let $m\ge2$ and $3\le d\le m+1$.  There is no spherical $2m$-design on $\Sph{d-1}$ with $D_m+1$ points.
\end{theorem}

\begin{proof}
The left side of \eqref{eq:signed-schoenberg} is nonnegative.  In \eqref{eq:A-m3} every factor except
\[
 q_m(d)=d^2+dm-2d-2m^2-2m
\]
is positive.  This is a convex quadratic in $d$ satisfying $q_m(m+1)=-(m+1)<0$ and $q_m(m+2)=2m>0$; its other root is negative because the product of the roots is $-2m(m+1)$.  Thus $q_m(d)<0$ exactly for positive integers $d\le m+1$, contradicting $A_{m+3,1}\ge0$.
\end{proof}

\begin{corollary}\label{cor:corank-one-s2}
For every $m\ge2$, no spherical $2m$-design on $\Sph{2}$ has $(m+1)^2+1$ points.
\end{corollary}

\begin{remark}\label{rem:complementary}
Theorem~\ref{thm:sign-obstruction}, together with Theorems~\ref{thm:strength4} and \ref{thm:strength6}, excludes corank one whenever $m\le3$ or $d\le m+1$.  In the complementary region $m\ge4$, $d\ge m+2$, the coefficient $A_{m+3,1}$ is positive, and the residue formula gives no further sign obstruction there either: an exact computation with \eqref{eq:residue-formula} (script \path{corank1_residue_explore.py}, Appendix~\ref{app:repro}) shows that for $4\le m\le10$ and $m+2\le d\le m+12$ every coefficient $A_{\ell,1}$ with $\ell\le m+24$ is positive, the minimum over that range of $\ell$ being attained at $\ell=m+3$.  The remaining obstruction is arithmetic: the forced multiplicities of Lemma~\ref{lem:distribution} must be nonnegative integers and are constant on Galois orbits.
\end{remark}

\section{Corank two in every dimension}\label{sec:all-dimensional-corank-two}

Throughout this section $X\subset\Sph{d-1}$ is a spherical $4$-design with
\[
 N=|X|=D_2(d)+2,
 \qquad
 D_2(d)=\frac{d(d+3)}2.
\]
Recall the evaluation spaces $\mathcal E_j(X)$ (Definition~\ref{def:evaluation}), the projective lift $\widehat X\subset\PP^d_\C$, its coordinate ring $R_X$, its Hilbert function $H_X(j)=\dim(R_X)_j=\dim\mathcal E_j(X)$, and the Artinian reduction $A=R_X/(z_0)$ (Definition~\ref{def:lift}).

\subsection{Hilbert function and cubic saturation}

\begin{lemma}[Hilbert function through degree two]\label{lem:hilbert-through-two}
One has
\[
 H_X(0)=1,\qquad H_X(1)=d+1,\qquad H_X(2)=D_2(d)=N-2.
\]
The degree-two kernel is spanned by $z_1^2+\cdots+z_d^2-z_0^2$.
\end{lemma}

\begin{proof}
If an affine polynomial $p$ of degree at most two vanishes on $X$, design exactness applied to $|p|^2$ gives $\int_{\Sph{d-1}}|p|^2\dd\omega=0$.  Thus $p$ vanishes on the sphere, and in degree at most two it is a scalar multiple of $|x|^2-1$.  The homogeneous statement and the dimensions follow.
\end{proof}

\begin{corollary}[Cubic saturation and quadratic Cayley--Bacharach]\label{lem:degree-three-saturation}
The cubic evaluation map is onto, $\mathcal E_3(X)=\C^X$, and $\widehat X$ is Cayley--Bacharach in degree two.
\end{corollary}

\begin{proof}
Apply Proposition~\ref{prop:cb-hilbert} with $m=2$ and corank $c=2$.
\end{proof}

Lemma~\ref{lem:hilbert-through-two} and Corollary~\ref{lem:degree-three-saturation} give
\begin{equation}\label{eq:artinian-hilbert-vector}
 \operatorname{Hilb}(A)=\left(1,d,\frac{(d-1)(d+2)}2,2\right),
 \qquad A_j=0\quad(j\ge4).
\end{equation}

\subsection{The Gale circle and linear--quadratic aliases}\label{sec:alias}

Let $W=\mathcal E_2(X)^\perp$.  Theorem~\ref{thm:finite-defect} with $c=2$ identifies its projection with the Gram matrix of a circle $2$-design.  Choosing a complex coordinate on that circle gives phases $\zeta=(\zeta_x)_{x\in X}$ with $|\zeta_x|=1$ and
\begin{equation}\label{eq:circleframe}
 W_\C=\operatorname{span}\{\zeta,\bar\zeta\},
 \qquad
 \sum_{x\in X}\zeta_x=0,
 \qquad
 \sum_{x\in X}\zeta_x^2=0.
\end{equation}
Hence
\[
 \C^X=\Harm_0^\C(X)\oplus\Harm_1^\C(X)\oplus\Harm_2^\C(X)
 \oplus\C\zeta\oplus\C\bar\zeta
\]
is an orthogonal decomposition, and
\begin{equation}\label{eq:general-gale-kernel}
 K_2^{(d)}(x\cdot y)=-2\Re(\zeta_x\bar\zeta_y)\qquad(x\ne y).
\end{equation}

\begin{theorem}[Rank-two linear--quadratic aliases]\label{thm:alias}
There is a subspace $U\subset\Harm_1^\C$ with $\dim_\C U\ge d-1$ such that for every $A\in U$ there is $\eta\in\Harm_2^\C$ satisfying
\[
 A(x)=\zeta_x\eta(x)\qquad(x\in X).
\]
We call such a pair $(A,\eta)$ a \emph{linear--quadratic alias}.
\end{theorem}

\begin{proof}
For $A,B\in\Harm_1^\C$,
\[
 \langle\bar\zeta A,1\rangle_X=0,
 \qquad
 \langle\bar\zeta A,B\rangle_X=0,
 \qquad
 \langle\bar\zeta A,\bar\zeta\rangle_X=0.
\]
The middle identity holds because $A\bar B\in\Harm_0^\C\oplus\Harm_2^\C$, which is orthogonal to $W_\C$; the last uses $|\zeta|=1$ and the vanishing first moment.  Therefore
\[
 \bar\zeta\Harm_1^\C\subset\Harm_2^\C(X)\oplus\C\zeta.
\]
The left side has dimension $d$, so its intersection with $\Harm_2^\C(X)$ has dimension at least $d-1$.  Multiplying back by $\zeta$ proves the theorem.
\end{proof}

For each coordinate there are unique $\eta_j\in\Harm_2^\C$ and $\lambda_j\in\C$ such that
\begin{equation}\label{eq:coordinate-alias}
 \bar\zeta x_j=\eta_j+\lambda_j\zeta\quad\text{on }X,
 \qquad
 \lambda_j=\frac1N\sum_{x\in X}\bar\zeta_x^{\,2}x_j.
\end{equation}
Put $\lambda=(\lambda_j)\in\C^d$, let $e_2\in\Harm_2^\C$ be the harmonic-quadratic component of $\zeta^2$, and set
\[
 \tau=\frac1N\sum_{x\in X}\zeta_x^3.
\]

\begin{theorem}[Exact corank-two identities]\label{thm:corank-two-identities}
Let $d\ge2$.
\begin{enumerate}
\item[(a)] The coordinate aliases satisfy
\begin{equation}\label{eq:alias-isometry}
 \langle \eta_i,\eta_j\rangle_{L^2}=\frac{\delta_{ij}}d-\lambda_i\bar\lambda_j.
\end{equation}
\item[(b)] On $X$,
\begin{equation}\label{eq:zeta-square}
 \zeta^2=d(\bar\lambda\cdot x)+e_2+\tau\bar\zeta,
\end{equation}
and
\begin{equation}\label{eq:zeta-square-norm}
 d|\lambda|^2+\|e_2\|_{L^2}^2+|\tau|^2=1.
\end{equation}
In particular $|\lambda|^2\le1/d$.
\item[(c)] One has $|\tau|<1$.
\end{enumerate}
\end{theorem}

\begin{proof}
Multiplication by $\bar\zeta$ is unitary on $\C^X$, and degree-four exactness identifies the sampled and spherical inner products on $\Harm_2^\C$.  Expanding \eqref{eq:coordinate-alias} therefore gives
\[
 \frac{\delta_{ij}}d=\langle\bar\zeta x_i,\bar\zeta x_j\rangle_X
 =\langle \eta_i,\eta_j\rangle_{L^2}+\lambda_i\bar\lambda_j,
\]
which is (a).

Decompose $\zeta^2$ in the orthogonal decomposition following \eqref{eq:circleframe}.  The constant and $\zeta$ components vanish because $\sum\zeta_x^2=\sum\zeta_x=0$.  The linear component is $d\bar\lambda\cdot x$, the quadratic component is $e_2$, and the $\bar\zeta$ component is $\tau\bar\zeta$.  This proves \eqref{eq:zeta-square}.  Degree-four exactness gives $\|e_2\|_X=\|e_2\|_{L^2}$, so Pythagoras proves \eqref{eq:zeta-square-norm}.

Nothing so far used the value of $d$.  If $|\tau|=1$, Pythagoras forces $\lambda=e_2=0$ and $\zeta^2=\tau\bar\zeta$, hence $\zeta^3=\tau$ at every node.  The phases then lie among three equally spaced values.  The identity $\sum\zeta_x=0$ forces the three multiplicities to be equal, so $3\mid N$.  But
\[
 N=\frac{d(d+3)}2+2\not\equiv0\pmod3,
\]
a contradiction.
\end{proof}

\begin{remark}
For $d=2$ the statements of Theorem~\ref{thm:corank-two-identities} are consistent with the regular heptagon; for $d=3$ they are used in Sections~\ref{sec:circle}--\ref{sec:moment}.
\end{remark}

\subsection{The quadratic socle}\label{sec:socle}

Define
\[
 \mathscr S_X=\{q\in\Harm_2^\C:x_jq|_X\in\mathcal E_2(X)\text{ for }1\le j\le d\}.
\]

\begin{proposition}[Socle interpolation]\label{prop:socle-interpolation}
Harmonic representatives identify $\Soc(A)_2$ with $\mathscr S_X$.  For every $q\in\mathscr S_X$ there are unique $Q_1,\ldots,Q_d\in\Pol_2$ satisfying
\[
 Q_j(x)=x_jq(x)\quad(x\in X),
\]
and
\[
 \sum_{j=1}^d\|Q_j\|_{L^2(\Sph{d-1})}^2=\|q\|_{L^2(\Sph{d-1})}^2.
\]
\end{proposition}

\begin{proof}
Because evaluation is injective on $\Pol_2$ and $\mathcal E_2(X)=\mathcal E_1(X)\oplus\Harm_2^\C(X)$, every class in $A_2$ has a unique harmonic-quadratic representative.  The socle condition is precisely $x_jq|_X\in\mathcal E_2(X)$ for each coordinate.  Uniqueness of the $Q_j$ follows.  Finally, design exactness through degree four gives
\[
 \sum_j\|Q_j\|_2^2
 =\frac1N\sum_{x\in X}|q(x)|^2\sum_jx_j^2
 =\frac1N\sum_{x\in X}|q(x)|^2
 =\|q\|_2^2.
\]
\end{proof}

\begin{theorem}[Forced quadratic socle]\label{thm:forced-quadratic-socle}
One has
\[
 \dim_\C\Soc(A)_2\ge\max\left\{0,\frac{d^2-3d-2}{2}\right\}.
\]
In particular $A$ is not level when $d\ge4$.
\end{theorem}

\begin{proof}
Multiplication gives
\[
 A_2\longrightarrow\operatorname{Hom}_\C(A_1,A_3),
 \qquad q\longmapsto(a\mapsto aq),
\]
whose kernel is exactly $\Soc(A)_2$.  By \eqref{eq:artinian-hilbert-vector}, the source has dimension $(d-1)(d+2)/2$ and the target has dimension $2d$.  Rank--nullity gives the result.
\end{proof}

\begin{proposition}[The socle through the aliases]\label{prop:socle-alias}
The quadratic socle is
\begin{equation}\label{eq:socle-alias-orthogonal}
 \mathscr S_X=\operatorname{span}\{\eta_1,\ldots,\eta_d,\bar\eta_1,\ldots,\bar\eta_d\}^{\perp}
 \cap\Harm_2^\C.
\end{equation}
Consequently
\[
 \dim\mathscr S_X=h_2-\rank\{\eta_j,\bar\eta_j:1\le j\le d\}
 \ge h_2-2d=\frac{d^2-3d-2}{2},
\]
which recovers the bound of Theorem~\ref{thm:forced-quadratic-socle}.
\end{proposition}

\begin{proof}
By Proposition~\ref{prop:socle-interpolation}, $q\in\mathscr S_X$ exactly when $x_jq$ is orthogonal to both $\zeta$ and $\bar\zeta$ for every $j$.  Conjugating \eqref{eq:coordinate-alias} and using $q\perp W_\C$ gives
\[
 \langle x_jq,\zeta\rangle_X=\langle q,\bar\eta_j\rangle_{L^2},
 \qquad
 \langle x_jq,\bar\zeta\rangle_X=\langle q,\eta_j\rangle_{L^2}.
\]
This is \eqref{eq:socle-alias-orthogonal}, and the dimension statement follows.
\end{proof}

\begin{remark}[Why ordinary Cayley--Bacharach is insufficient]\label{rem:quadratic-socle}
Quadratic Cayley--Bacharach excludes a separator supported at one node, but it does not exclude a general degree-two socle element.  Theorem~\ref{thm:forced-quadratic-socle} shows that such elements are forced for $d\ge4$.  Any higher-dimensional classification must therefore use the spherical evaluation and multiplication structure, not an unsupported inference from Cayley--Bacharach to levelness.
\end{remark}

\begin{remark}[Where the three-dimensional proof stops]\label{rem:higher-stop}
With $d-1$ aliases, the matrix $HH^*-r^2AA^*$ is $(d-1)\times(d-1)$ but has rank at most two, so its full determinant vanishes for $d\ge4$.  A two-alias compression has determinant $-r^2|\det(A,H)|^2$; even in a branch where the cubic determinant can be made real, its intersection with a scalar quartic in $\PP^{d-1}$ has dimension $d-3$ rather than zero.  Thus the length-twelve Cayley--Bacharach argument below is genuinely three-dimensional.  Theorem~\ref{thm:corank-two-identities} records the exact identities that survive in every dimension; closing $d\ge4$ requires additional polynomial-ring structure.
\end{remark}

\subsection{The circle}

\begin{proposition}[The regular heptagon]\label{prop:regular-heptagon}
Every seven-point spherical $4$-design on $\Sph{1}$ is a regular heptagon, up to rotation.
\end{proposition}

This is a special case of Hong's circle classification \cite{Hong82}; we include the short proof.
\begin{proof}
Identify $\Sph{1}$ with the unit complex numbers and write the nodes as $z_1,\ldots,z_7$.  Exactness gives $\sum_i z_i^j=0$ for $1\le j\le4$.  Newton's identities give $e_1=\cdots=e_4=0$, and $|z_i|=1$ gives $e_{7-j}=e_7\overline{e_j}$, hence $e_5=e_6=0$.  Thus $\prod_i(z-z_i)=z^7-e_7$ with $|e_7|=1$, so the nodes form a regular heptagon.
\end{proof}

We now specialise to $d=3$ and $N=11$.

\section{The eleven-point Gale circle}\label{sec:circle}

The goal of Sections~\ref{sec:circle}--\ref{sec:moment} is the following theorem.

\begin{theorem}\label{thm:eleven}
There is no eleven-point equal-weight spherical $4$-design on $\Sph{2}$.
\end{theorem}

\medskip
\noindent\textbf{Standing assumption for Sections~\ref{sec:circle}--\ref{sec:moment}.}
Assume for contradiction that
\[
X=\{x_1,\dots,x_{11}\}\subset\Sph{2}
\]
is a spherical $4$-design; so $d=3$, $m=2$, $N=11$, $D_2=9$ and the corank is $c=2$.  Write $t_{ij}=x_i\cdot x_j$.  The proof is organised as follows: Section~\ref{sec:circle} sets up the Gale circle and proves a line bound and a conic bound; Section~\ref{sec:cubic} produces a nonzero real harmonic cubic $R$ vanishing at the nodes; Section~\ref{sec:hermitian} packages the aliases into a Hermitian quartic matrix $\FF$ with $\det\FF=-r^2R^2$; Section~\ref{sec:generic} excludes the branch in which some scalar compression of $\FF$ is coprime to $R$ (Proposition~\ref{prop:generic}); Sections~\ref{sec:pauli}--\ref{sec:moment} exclude the complementary branch $\FF=R\LL$ (Proposition~\ref{prop:exceptional}). The degree-two evaluation space has dimension nine, and its complement has rank two. The circle-frame representation \eqref{eq:circleframe} gives phases $\zeta_i\in\Sph{1}\subset\C$ satisfying
\begin{equation}\label{eq:zetasums}
\sum_{i=1}^{11}\zeta_i=0,\qquad \sum_{i=1}^{11}\zeta_i^2=0,
\end{equation}
and, for $i\ne j$,
\begin{equation}\label{eq:gale-kernel}
\Re(\zeta_i\bar\zeta_j)=\kappa(t_{ij}),\qquad \kappa(t)=\frac{3-6t-15t^2}{4}.
\end{equation}
Indeed, the complementary projection is $(2/11)\Gamma$, where $\Gamma_{ij}=\Re(\zeta_i\bar\zeta_j)$, and $P_{ij}=K(t_{ij})/11$ with $K$ from \eqref{eq:K2d3}.

Because $\kappa(1)=-9/2$ and $\kappa(-1)=-3/2$, no two nodes are equal or antipodal. Thus they determine eleven distinct points of $\PP^2$.

\subsection{Line and conic bounds}

\begin{lemma}[Line bound]\label{lem:line}
A real projective line contains at most four nodes of $X$.
\end{lemma}

\begin{proof}
Let $\ell(x)=u\cdot x$ be nonzero. Exactness gives
\[
\sum_i\ell(x_i)^2=\frac{11}3|u|^2,\qquad \sum_i\ell(x_i)^4=\frac{11}5|u|^4 .
\]
If $N_0$ of the values are nonzero, Cauchy--Schwarz gives
\[
N_0\ge\frac{(\sum_i\ell(x_i)^2)^2}{\sum_i\ell(x_i)^4}=\frac{55}9>6 .
\]
Hence at least seven values are nonzero.
\end{proof}

\begin{lemma}[Conic bound]\label{lem:conic}
A nonzero real projective conic contains at most eight nodes of $X$.
\end{lemma}

\begin{proof}
Let $\Psi$ be a nonzero real homogeneous quadratic vanishing at at least nine nodes, and put $\psi_i=\Psi(x_i)$. Then the evaluation vector $\psi=(\psi_i)$ is supported on at most two coordinates, and it is nonzero because evaluation is injective on $\Pol_2$. Since a quadratic evaluation lies in $\Harm_0\oplus\Harm_2$, it is orthogonal to the Gale plane, and hence
\[
\sum_i\psi_i\bar\zeta_i=0 .
\]
A one-point support is impossible. For a two-point support, equality of moduli gives $\zeta_j=\pm\zeta_i$.

If $\zeta_j=\zeta_i$, then \eqref{eq:gale-kernel} gives $\kappa(t_{ij})=1$, but the equation $\kappa(t)=1$ has no real root. If $\zeta_j=-\zeta_i$, then $\psi_i=\psi_j$. Exactness for the odd vector-valued cubic $\Psi(x)x$ gives
\[
0=\sum_k\Psi(x_k)x_k=\psi_i(x_i+x_j),
\]
so $x_j=-x_i$, already excluded.
\end{proof}

\section{Two aliases and the harmonic cubic}\label{sec:cubic}

By Theorem~\ref{thm:alias} with $d=3$, there are independent complex linear harmonics $A_1,A_2$ and complex harmonic quadratics $\eta_1,\eta_2$ such that
\begin{equation}\label{eq:alias}
A_j(x_i)=\zeta_i\eta_j(x_i),\qquad j=1,2,
\end{equation}
for every node. Define
\begin{equation}\label{eq:C}
C=A_1\eta_2-A_2\eta_1 .
\end{equation}
Then $C(x_i)=0$ for all $i$.

\begin{lemma}\label{lem:harmonic}
The homogeneous cubic $C$ is harmonic.
\end{lemma}

\begin{proof}
Write $C=C_3+r^2C_1$ with $C_3\in\Harm_3^\C$ and $C_1\in\Harm_1^\C$, where $r^2=x^2+y^2+z^2$. For every $M\in\Harm_1^\C$, design exactness gives
\[
0=\frac1{11}\sum_iC(x_i)\overline{M(x_i)}=\int_{\Sph{2}}C\bar M\dd\omega .
\]
The pairing vanishes on $\Harm_3^\C\times\Harm_1^\C$ and is nondegenerate on $r^2\Harm_1^\C\times\Harm_1^\C$. Hence $C_1=0$.
\end{proof}

\begin{lemma}[Reality dichotomy]\label{lem:reality}
Either $C=0$, or $C=e^{i\phi}R$ for a nonzero real harmonic cubic $R$.
\end{lemma}

\begin{proof}
Write $C=P+iQ$ with real harmonic cubics $P,Q$. Suppose $P,Q$ are independent. Since they have eleven common projective zeros, B\'ezout forces a nonconstant common divisor. Let $J$ be a greatest common divisor of $P$ and $Q$ in $\C[x,y,z]$. Complex conjugation of coefficients maps a greatest common divisor to a greatest common divisor, so $\bar J=cJ$ with $|c|=1$, and a unimodular rescaling makes $J$ real; we take $J$ real.

If $\deg J=1$, the residual quadratics $P/J$, $Q/J$ are coprime and meet in at most four points, so at least seven nodes lie on the common line, contradicting Lemma~\ref{lem:line}. If $\deg J=2$, the residual lines meet in at most one point, so at least ten nodes lie on the common conic, contradicting Lemma~\ref{lem:conic}. If $\deg J=3$, the cubics are proportional. Thus $P,Q$ are proportional in all possible cases, and the conclusion follows.
\end{proof}

\begin{lemma}\label{lem:Cnonzero}
The case $C=0$ is impossible.
\end{lemma}

\begin{proof}
If $C=0$, then $A_1\eta_2=A_2\eta_1$. The independent linear forms $A_1,A_2$ are coprime, so there is a complex linear form $M$ with
\[
\eta_1=A_1M,\qquad \eta_2=A_2M .
\]
At a node where $(A_1,A_2)\ne(0,0)$, \eqref{eq:alias} gives $M(x_i)=\bar\zeta_i$. The two independent projective lines have at most one common projective point, so this holds at at least ten nodes. Therefore the real quadratic form
\[
M\bar M-r^2
\]
vanishes at at least ten nodes. By Lemma~\ref{lem:conic}, it must vanish identically. Writing $M=a\cdot x+i\,b\cdot x$ with $a,b\in\R^3$, however,
\[
M\bar M=(a\cdot x)^2+(b\cdot x)^2
\]
has real quadratic rank at most two, while $r^2=x^2+y^2+z^2$ has rank three.
\end{proof}

After multiplying one alias pair by a constant phase, we henceforth assume
\begin{equation}\label{eq:CR}
C=R\in\Harm_3(\R),\qquad R\ne0 .
\end{equation}

\section{The Hermitian quartic matrix}\label{sec:hermitian}

Set
\[
A=\begin{pmatrix}A_1\\A_2\end{pmatrix},\qquad H=\begin{pmatrix}\eta_1\\\eta_2\end{pmatrix},
\]
and define the Hermitian quartic matrix
\begin{equation}\label{eq:F}
\FF=HH^*-r^2AA^* ,
\end{equation}
where $^*$ denotes conjugate transpose with conjugation acting on coefficients; for real $x$ the value $\FF(x)$ is a Hermitian $2\times2$ matrix. Every entry vanishes at every node, because \eqref{eq:alias} gives $A=\zeta H$ there. Moreover, for any two vectors $u,v\in\C^2$,
\[
\det(uu^*-vv^*)=-|\det(u,v)|^2 .
\]
Applying this with $u=H$ and $v=rA$ yields
\begin{equation}\label{eq:detF}
\det\FF=-r^2|\det(A,H)|^2=-r^2R^2 .
\end{equation}

We divide the proof into a generic branch and an exceptional branch. Both branches use changes of alias basis $A\mapsto SA$, $H\mapsto SH$ with $S\in SL_2(\C)$; these preserve the alias relation $A=\zeta H$ at the nodes, transform $\FF$ by the congruence $\FF\mapsto S\FF S^*$, and preserve $R=\det(A,H)$ because $\det S=1$.

\section{The generic branch: matrix Cayley--Bacharach}\label{sec:generic}

For $v\in\C^2$, put
\[
f_v=v^*\FF v .
\]
This is a real homogeneous quartic vanishing on the eleven nodes.

\begin{lemma}[Common-factor dichotomy]\label{lem:dichotomy}
For every nonzero $v$, either $\gcd(R,f_v)=1$ or $R\mid f_v$.
\end{lemma}

\begin{proof}
If $R$ and $f_v$ have a nonconstant common factor but $R\nmid f_v$, their greatest common divisor, which may be taken real as in the proof of Lemma~\ref{lem:reality}, has degree one or two. In degree one, the residual curves have degrees two and three and meet in at most six nodes, so at least five nodes lie on the common line, contradicting Lemma~\ref{lem:line}. In degree two, the residual curves have degrees one and two and meet in at most two points, so at least nine nodes lie on the common conic, contradicting Lemma~\ref{lem:conic}.
\end{proof}

\begin{proposition}[The generic branch is impossible]\label{prop:generic}
It is impossible that some scalar compression $f_v$ fails to be divisible by $R$.
\end{proposition}

\begin{proof}
Assume that some $f_v$ is not divisible by $R$. The map $v\mapsto f_v$ is a Hermitian form in $v$ with values in the space of real quartics, and divisibility by $R$ is a linear condition on quartics; hence the set of $v$ for which $R\mid f_v$ is a proper real-algebraic subset of $\C^2\cong\R^4$. For each node with $A(x_i)\ne0$, the condition $v^*A(x_i)=0$ defines a proper real-algebraic subset as well. We may therefore choose $v$ so that
\[
\gcd(R,f_v)=1
\]
and, for each node with $A(x_i)\ne0$,
\[
v^*A(x_i)\ne0 .
\]
After an $SU_2$ change of alias basis taking $v/|v|$ to the first basis vector, write
\begin{equation}\label{eq:Fblock}
\FF=\begin{pmatrix}f&g\\\bar g&e\end{pmatrix},\qquad \gcd(f,R)=1,
\end{equation}
with $f,e$ real quartics and $g$ a complex quartic. Equation \eqref{eq:detF} becomes
\begin{equation}\label{eq:GG}
g\bar g=fe+r^2R^2 .
\end{equation}
Let
\[
Z=V(R,f)\subset\PP^2_\C .
\]
Because $\gcd(R,f)=1$, the pair $R,f$ is a homogeneous regular sequence in $\C[x,y,z]$.  Hence $\C[x,y,z]/(R,f)$ is a one-dimensional Cohen--Macaulay graded ring; it has no irrelevant-ideal torsion, so $(R,f)$ is saturated.  B\'ezout's theorem gives $\operatorname{length}Z=3\cdot4=12$.  The scheme contains the eleven distinct nodes. Since their reduced union has length eleven, exactly one unit of residual length remains: either a twelfth point, or one node carrying local length two.

If the residual length is supported at a new point, all twelve points are simple. Cayley--Bacharach for a $(3,4)$ complete intersection says that a quartic passing through eleven of them passes through the twelfth \cite{EGH96}; applying this to the real and imaginary parts of $g$, the quartic $g$ vanishes on all of $Z$.

Suppose instead that one node $p$ has local length two and all other nodes are simple.  Let $\mathcal O$ be the local ring of the smooth surface $\PP^2_\C$ at $p$, with maximal ideal $\mathfrak m$.  A local quotient of $\mathcal O$ of length two has one-dimensional maximal ideal with square zero; equivalently, its defining ideal contains $\mathfrak m^2$ and a codimension-one subspace of $\mathfrak m/\mathfrak m^2$.  Thus the local scheme is curvilinear and determines a unique tangent direction.  Since the local component of the real scheme $Z$ at the real point $p$ is conjugation-invariant, this tangent direction may be represented by a real vector $w$. The tangent direction is nonzero in the projective tangent space, so $w\notin\C p$. A polynomial vanishes on the local scheme exactly when its value at $p$ is zero and its differential kills $w$. Since $R,f\in I(Z)$,
\begin{equation}\label{eq:tangent}
dR_p(w)=df_p(w)=0 .
\end{equation}
Because $R$ and $f$ are homogeneous, $dR_p(p)=3R(p)=0$ and $df_p(p)=4f(p)=0$, so we may replace $w$ by its component orthogonal to $p$ and assume $w\perp p$; then $d(r^2)_p(w)=2p\cdot w=0$.

At $p$, write $H(p)=\vartheta A(p)$ with $|\vartheta|=1$ (indeed $\vartheta=\bar\zeta_p$ by \eqref{eq:alias}). If $A(p)=0$, then also $H(p)=0$, and every first derivative of $\FF$ vanishes at $p$. Assume $A(p)\ne0$ and set
\[
V=dH_p(w)-\vartheta\,dA_p(w)\in\C^2 .
\]
Differentiating $R=\det(A,H)$ and using \eqref{eq:tangent} gives $\det(A(p),V)=0$, so $V=cA(p)$ for some $c\in\C$. Differentiation of \eqref{eq:F}, with $d(r^2)_p(w)=0$, gives
\begin{equation}\label{eq:dF}
d\FF_p(w)=2\Re(\bar\vartheta c)\,A(p)A(p)^* .
\end{equation}
The first diagonal compression is $f=f_v$, and our choice gives $v^*A(p)\ne0$, i.e.\ $A_1(p)\ne0$ after the basis change. Thus $df_p(w)=0$ forces $\Re(\bar\vartheta c)=0$, so $d\FF_p(w)=0$. In particular $dg_p(w)=0$. Hence $g$ vanishes on the full length-two local scheme as well as at the other ten nodes.

In either residual case, $g$ vanishes scheme-theoretically on $Z$, and therefore
\[
g\in I(Z)=(R,f).
\]
Taking the degree-four piece,
\begin{equation}\label{eq:Gdecomp}
g=af+\mu R
\end{equation}
for a constant $a\in\C$ and a complex linear form $\mu$. Reducing \eqref{eq:GG} modulo $f$ yields
\[
f\mid(\mu\bar\mu-r^2)R^2 .
\]
Since $\gcd(f,R)=1$,
\[
f\mid\mu\bar\mu-r^2 .
\]
The right side has degree two, so it must vanish identically. This is impossible by the same rank-two versus rank-three argument used in Lemma~\ref{lem:Cnonzero}. Thus the generic branch cannot occur.
\end{proof}

\section{The exceptional branch: Pauli factorisation}\label{sec:pauli}

\begin{proposition}[The exceptional branch is impossible]\label{prop:exceptional}
It is impossible that every scalar compression $f_v$, $v\in\C^2$, is divisible by $R$.
\end{proposition}

The proof of Proposition~\ref{prop:exceptional} occupies the rest of this section and Sections~\ref{sec:spinor}--\ref{sec:moment}.  Assume that every $f_v$ is divisible by $R$. Polarisation (take $v=e_1,e_2,e_1+e_2,e_1+ie_2$) then gives entrywise divisibility:
\begin{equation}\label{eq:FRL}
\FF=R\LL ,
\end{equation}
where $\LL$ is a $2\times2$ matrix of linear forms which is Hermitian at real arguments. From \eqref{eq:detF},
\begin{equation}\label{eq:detL}
\det\LL=-r^2 .
\end{equation}

\begin{lemma}[Pauli normalisation]\label{lem:pauli}
After an $SL_2(\C)$ congruence in alias space and an orthogonal change of physical coordinates,
\begin{equation}\label{eq:pauliL}
\LL(x,y,z)=\begin{pmatrix}z&x-iy\\x+iy&-z\end{pmatrix}.
\end{equation}
\end{lemma}

\begin{proof}
Let $\Herm$ be the four-dimensional real vector space of Hermitian $2\times2$ matrices with determinant form $\delta(M)=\det M$. This form has signature $(1,3)$. Equation \eqref{eq:detL} says that the linear map
\[
\Phi:\R^3\to\Herm,\qquad x\mapsto\LL(x),
\]
identifies $\R^3$ with a maximal negative-definite subspace.

Its $\delta$-orthogonal complement is a positive line. Choose $T$ on that line with $T$ positive definite and $\det T=1$. Put $S=T^{-1/2}$. Then $S\in SL_2(\C)$, congruence by $S$ preserves determinant, and $STS^*=I$. Consequently $S\Phi(\R^3)S^*$ is $\delta$-orthogonal to $I$, hence lies in the traceless Hermitian subspace (the polar form of $\det$ against $I$ is $\tfrac12\tr$). Every traceless Hermitian matrix has a unique expression $u_1\sigma_1+u_2\sigma_2+u_3\sigma_3$, and its determinant is $-|u|^2$. Thus the transformed coefficient map from physical $\R^3$ to $u$-space is orthogonal. Absorbing that orthogonal map into the physical coordinates gives \eqref{eq:pauliL}.
\end{proof}

The transformations in Lemma~\ref{lem:pauli} preserve the alias relation $A=\zeta H$ at the nodes and preserve $R=\det(A,H)$ because the alias-space congruence has determinant one; the orthogonal change of coordinates moves the design to a congruent design.

For $u=(u_1,u_2)^T$, define
\[
u^\perp=(-\bar u_2,\bar u_1)^T .
\]
Then $A^*A^\perp=0$ and
\[
H^*A^\perp=\overline{\det(A,H)}=R .
\]
Multiplying \eqref{eq:FRL} by $A^\perp$ gives $R H=R\LL A^\perp$, hence
\begin{equation}\label{eq:HLA}
H=\LL A^\perp .
\end{equation}

\section{Classification of the harmonic Pauli branch}\label{sec:spinor}

Write $A=Bx$, where $x=(x,y,z)^T$ and $B\in M_{2\times3}(\C)$.

\begin{lemma}[Harmonic spinor parameterisation]\label{lem:spinor}
The two components of $H=\LL A^\perp$ are harmonic quadratics if and only if
\begin{equation}\label{eq:Bform}
B=\begin{pmatrix}p&q&-s+it\\ s&t&p+iq\end{pmatrix}
\end{equation}
for complex parameters $p,q,s,t$.
\end{lemma}

\begin{proof}
Write
\[
A_1=px+qy+\rho z,\qquad A_2=sx+ty+uz .
\]
Using \eqref{eq:pauliL} and \eqref{eq:HLA}, direct differentiation gives
\[
\Delta H_1=2(\bar p-i\bar q-\bar u),\qquad \Delta H_2=-2(\bar s+i\bar t+\bar\rho).
\]
Thus harmonicity is equivalent to $u=p+iq$ and $\rho=-s+it$.
\end{proof}

Let $\sigma_1,\sigma_2,\sigma_3$ be the Pauli matrices, so $\LL=x\sigma_1+y\sigma_2+z\sigma_3$. Since $H=\LL A^\perp$,
\begin{equation}\label{eq:RALA}
R=\det(A,H)=-A^*\LL A .
\end{equation}
Substitution of \eqref{eq:Bform} gives the exact formula
\begin{equation}\label{eq:laplaceR}
\Delta R=-2\sum_{\alpha=1}^3\tr(BB^*\sigma_\alpha)\,x_\alpha .
\end{equation}
Since $R$ is harmonic and the Pauli matrices span the traceless Hermitian matrices,
\begin{equation}\label{eq:BBstar}
BB^*=\lambda I_2
\end{equation}
for some $\lambda>0$.

\begin{lemma}\label{lem:BstarB}
Under \eqref{eq:Bform}, equation \eqref{eq:BBstar} implies that $B^*B$ is a real symmetric rank-two matrix.
\end{lemma}

\begin{proof}
A direct coefficient calculation gives
\begin{gather*}
\Im(B^*B)_{12}=\tfrac12\bigl((BB^*)_{11}-(BB^*)_{22}\bigr),\\
\Im(B^*B)_{13}=\Im(BB^*)_{12},\qquad
\Im(B^*B)_{23}=\Re(BB^*)_{12}.
\end{gather*}
All three vanish under \eqref{eq:BBstar}. Since the two rows of $B$ are independent, $B^*B$ has rank two.
\end{proof}

The one-dimensional kernel of $B^*B$ is therefore spanned by a real vector. The Pauli pencil is equivariant under the double cover $SU_2\to SO_3$: if $O\in SO_3$ and $U\in SU_2$ satisfy
\[
U\LL(O^Tx)U^*=\LL(x),
\]
then
\[
A'(x)=UA(O^Tx),\qquad H'(x)=UH(O^Tx)
\]
still satisfy $H'=\LL(A')^\perp$, and $B'=UBO^T$. We may therefore rotate the real kernel of $B^*B$ to the $y$-axis. In \eqref{eq:Bform} this gives $q=t=0$ and
\begin{equation}\label{eq:B5}
B=\begin{pmatrix}p&0&-s\\ s&0&p\end{pmatrix}.
\end{equation}
Equation \eqref{eq:BBstar} says $p\bar s-s\bar p=0$, so $p$ and $s$ have a common complex phase (trivially so if one of them vanishes). The transformation
\[
A\mapsto e^{-i\phi}A,\qquad H\mapsto e^{i\phi}H,\qquad \zeta\mapsto e^{-2i\phi}\zeta
\]
removes this phase without changing $R$, $\FF$, $\LL$, \eqref{eq:HLA} or \eqref{eq:zetasums}--\eqref{eq:gale-kernel}. Thus $p,s$ may be taken real.

\begin{lemma}[Spin-$3/2$ rotation]\label{lem:spin}
Let $O_\theta\in SO_3$ be the rotation about the $y$-axis by the angle $\theta$, with $O_\theta^T(x,y,z)^T=(x\cos\theta-z\sin\theta,\ y,\ x\sin\theta+z\cos\theta)^T$, and let
\[
U_\theta=\begin{pmatrix}\cos(\theta/2)&-\sin(\theta/2)\\ \sin(\theta/2)&\cos(\theta/2)\end{pmatrix}\in SU_2 .
\]
Then $U_\theta\LL(O_\theta^Tx)U_\theta^*=\LL(x)$, and for real $p,s$ the transformation $A'(x)=U_\theta A(O_\theta^Tx)$ applied to $A=Bx$ with $B$ as in \eqref{eq:B5} gives $A'=B'x$ with $B'$ of the same form and
\[
\begin{pmatrix}p'\\s'\end{pmatrix}=
\begin{pmatrix}\cos(3\theta/2)&-\sin(3\theta/2)\\ \sin(3\theta/2)&\cos(3\theta/2)\end{pmatrix}
\begin{pmatrix}p\\s\end{pmatrix}.
\]
\end{lemma}

\begin{proof}
The covariance identity is the standard spin lift of a rotation about the $y$-axis; it is verified by expanding $U_\theta\sigma_\alpha U_\theta^*$ for $\alpha=1,2,3$ ($\sigma_2$ is fixed, and $\sigma_1,\sigma_3$ are rotated into each other by the angle $\theta$).  For the second statement write the components of $A=Bx$ as $A_1+iA_2=(p+is)(x+iz)$ and $A_1-iA_2=(p-is)(x-iz)$, which is exactly the content of \eqref{eq:B5} for real $p,s$.  Under $x\mapsto O_\theta^Tx$ one has $x+iz\mapsto e^{i\theta}(x+iz)$, and $U_\theta$ acts on $(A_1,A_2)$ by $A_1+iA_2\mapsto e^{i\theta/2}(A_1+iA_2)$.  Hence $p+is\mapsto e^{3i\theta/2}(p+is)$, which is the displayed rotation.
\end{proof}

The transformation of Lemma~\ref{lem:spin} preserves the relation $H=\LL A^\perp$, by the equivariance recorded before \eqref{eq:B5}, and it moves the design to a congruent design.  Choose $\theta$ such that the transformed pair equals $(\sqrt{p^2+s^2},\,0)$. Dividing both $A$ and $H$ by this positive number preserves the alias relation and the factorisation $\FF=R\LL$ (both sides scale quadratically). We reach the canonical form
\begin{equation}\label{eq:Acanon}
A=\begin{pmatrix}x\\z\end{pmatrix}.
\end{equation}
Equation \eqref{eq:HLA} then gives
\begin{equation}\label{eq:Hcanon}
H=\begin{pmatrix}x^2-ixy-z^2\\ -z(2x+iy)\end{pmatrix}
\end{equation}
and
\begin{equation}\label{eq:Rcanon}
R=\det(A,H)=z^3-3x^2z=z(z^2-3x^2).
\end{equation}

\section{The final moment contradiction}\label{sec:moment}

Every node lies on $R=0$. At a node with $A\ne0$, the relation $A=\zeta H$ determines the phase.

If $z=0$, then \eqref{eq:Hcanon} gives $H=(x(x-iy),0)^T$, and hence, using $x^2+y^2=1$,
\begin{equation}\label{eq:zeta1}
\zeta=\frac1{x-iy}=x+iy .
\end{equation}
If $z=\pm\sqrt3\,x$, then $H=-(2x+iy)A$, and $4x^2+y^2=1$, so
\begin{equation}\label{eq:zeta2}
\zeta=-\frac1{2x+iy}=-2x+iy .
\end{equation}
Thus at every node with $A\ne0$,
\begin{equation}\label{eq:Imzeta}
\Im\zeta_i=y_i .
\end{equation}
If $A=0$, then $x=z=0$ and $y_i^2=1\ge(\Im\zeta_i)^2$. Therefore
\begin{equation}\label{eq:ineq}
\sum_{i=1}^{11}y_i^2\ge\sum_{i=1}^{11}(\Im\zeta_i)^2 .
\end{equation}
The spherical $2$-design identity gives
\begin{equation}\label{eq:y2}
\sum_iy_i^2=\frac{11}3 .
\end{equation}
On the other hand, $|\zeta_i|=1$ and $\sum_i\zeta_i^2=0$ imply
\begin{equation}\label{eq:Im2}
\sum_i(\Im\zeta_i)^2=\frac{11}2 .
\end{equation}
Equations \eqref{eq:ineq}--\eqref{eq:Im2} give $11/3\ge11/2$, a contradiction. This completes the proof of Proposition~\ref{prop:exceptional}.

\begin{proof}[Proof of Theorem~\ref{thm:eleven}]
Under the standing assumption, Sections~\ref{sec:circle}--\ref{sec:hermitian} produce the nonzero real harmonic cubic $R$ and the Hermitian quartic matrix $\FF$.  Either some scalar compression $f_v$ is not divisible by $R$, which Proposition~\ref{prop:generic} excludes, or all of them are, which Proposition~\ref{prop:exceptional} excludes.  Hence no eleven-point $4$-design exists on $\Sph{2}$.
\end{proof}

\section{The exact minimum on the two-sphere}\label{sec:minimum}

For completeness, the tight cardinality nine also has a short kernel proof.

\begin{lemma}\label{lem:nine}
There is no nine-point spherical $4$-design on $\Sph{2}$.
\end{lemma}

\begin{proof}
If $|X|=9$, the degree-two evaluation projection is the identity. Hence for distinct nodes $K(t_{ij})=0$, with $K$ from \eqref{eq:K2d3}. The two possible inner products are
\[
a,b=\frac{-1\pm\sqrt6}5 .
\]
Fix $x_i$, and let $k$ be the number of the other eight nodes having inner product $a$ with $x_i$. Since a $4$-design has zero centroid,
\[
1+ka+(8-k)b=0 .
\]
Thus $k=4+\sqrt6/4$, not an integer.
\end{proof}

The regular icosahedron is a spherical $5$-design and hence a spherical $4$-design with twelve points \cite{BB09,BD1,DGS}.  Combining Lemma~\ref{lem:nine} ($N=9$), Corollary~\ref{cor:ten} ($N=10$; independently Theorem~\ref{thm:sign-obstruction} with $(d,m)=(3,2)$) and Theorem~\ref{thm:eleven} ($N=11$) with the Fisher bound $N\ge9$ gives the exact value and proves Theorem~\ref{thm:main}.

\begin{theorem}[Exact minimum]\label{thm:exact}
The minimum cardinality of an equal-weight spherical $4$-design on $\Sph{2}$ is
\[
\boxed{N_4(\Sph{2})=12}\,.
\]
\end{theorem}

\section{Stability and open problems}\label{sec:further}

Let $(Y_{\ell r})$ be an orthonormal harmonic basis on $\Sph{2}$ and define the degree-four defect energy of an ordered $N$-tuple $X=(x_1,\ldots,x_N)$ by
\[
 \Phi_4(X)=\sum_{\ell=1}^4\sum_{r=1}^{2\ell+1}
 \left|\frac1N\sum_{j=1}^NY_{\ell r}(x_j)\right|^2.
\]
It vanishes exactly when the ordered tuple, with multiplicities, integrates every polynomial of degree at most four.

\begin{corollary}[Qualitative stability]\label{cor:stability}
For every $1\le N\le11$ there is a constant $c_N>0$ such that every ordered equal-weight $N$-point configuration on $\Sph{2}$, repetitions allowed, satisfies $\Phi_4(X)\ge c_N$.
\end{corollary}

\begin{proof}
Suppose $\Phi_4(X)=0$.  Let $E$ be the $N\times9$ normalised evaluation matrix of an orthonormal basis of $\Pol_2(\Sph{2})$.  Then $E^*E=I_9$, so $N\ge9$.  Put $P=EE^*$ and $Q=I-P$.  Since the reproducing kernel equals $9$ on the diagonal,
\[
 P_{ii}=\frac9N,
 \qquad
 Q_{ii}=\frac{N-9}{N}.
\]
If $x_i=x_j$ with $i\ne j$, the corresponding rows of $E$ coincide, so $Q_{ij}=-9/N$.  Positivity of the corresponding $2\times2$ principal minor gives
\[
 0\le\frac{(N-9)^2-81}{N^2},
\]
which for $N\ge9$ forces $N\ge18$.  Thus a zero with $N\le11$ has distinct nodes and would contradict Theorem~\ref{thm:exact}.  Compactness of $(\Sph{2})^N$ and continuity of $\Phi_4$ give a positive minimum.
\end{proof}

Three problems remain especially natural.  First, determine whether spherical $4$-designs on $\Sph{2}$ exist with $13$ or $15$ nodes.  Second, settle corank two in dimensions $d\ge4$; Theorem~\ref{thm:corank-two-identities} and the forced socle identify the structure that any example must carry.  Third, classify corank one in the complementary range $m\ge4$, $d\ge m+2$, where the residue sign obstruction disappears (Remark~\ref{rem:complementary}) and the arithmetic of $K_m^{(d)}\pm1$ becomes decisive.

Numerical searches may guide these problems, but no finite truncation is used as a theorem here.  In particular, feasibility of finitely many Schoenberg inequalities does not by itself prove feasibility of the full infinite positivity programme.

\appendix

\section{Corank-one coefficient extraction}\label{app:identities}

This appendix records the finite algebra behind the closed forms \eqref{eq:A-m1}--\eqref{eq:A-m3}.

\label{app:residue-coefficients}
In the notation of Theorem~\ref{thm:residue}, put
\[
 \mathcal R(s)=
 \frac{(1+\gamma_1s+\gamma_2s^2)(1+\pi_2s^2)}
 {(1-s)(1+\kappa_1s+\kappa_2s^2)^2}.
\]
The coefficients needed for $j=1,2,3$ are
\begin{align*}
 [s^0]\mathcal R&=1,\\
 [s^1]\mathcal R&=1+\gamma_1-2\kappa_1
 =\frac{(d-2)(d+m-1)}{d(d+2m-2)},\\
 [s^2]\mathcal R&=\gamma_2+\pi_2+1+3\kappa_1^2-2\kappa_2
 +\gamma_1-2\kappa_1\gamma_1-2\kappa_1\\
 &=\frac{(d-2)(d+m-1)(d^2+dm-2d-2m^2-2m)}
 {d(d+2m-2)^2(d+2m+2)}.
\end{align*}
Together with \eqref{eq:lead-ratio}, these are exactly \eqref{eq:A-m1}--\eqref{eq:A-m3}.

\section{Computational certificates for Section~\ref{sec:higher}}\label{app:certificates}

\subsection{Rational roots at strength six: proof of Lemma~\ref{lem:rational-roots}}\label{app:rational-roots}
Write $g=g_d$. Every real root of $g$ lies in $(-1,1)$. Indeed,
\[
 g'(t)=3d\bigl((d+4)t^2+2t-1\bigr)>0\qquad(|t|\ge1),
\]
while
\[
 g(-1)=-d^2+2d-3<0,\qquad g(1)=d^2+4d-3>0.
\]
Thus $g(t)<0$ for $t\le-1$ and $g(t)>0$ for $t\ge1$. By the rational root theorem, a rational root $t=\varepsilon/s$ in lowest terms has $\varepsilon\mid3$, so $\varepsilon\in\{\pm1,\pm3\}$, $s\ge1$, $\gcd(|\varepsilon|,s)=1$, and $|\varepsilon/s|<1$. Substituting and clearing $s^3$ gives
\[
E_\varepsilon(d,s):=d(d+4)\varepsilon^3+3d\varepsilon^2s-3d\varepsilon s^2-3s^3=0,
\]
a quadratic in $d$ with integer coefficients whose discriminant $\Delta_\varepsilon(s)$ must be a perfect square:
\begin{align*}
\Delta_{\pm1}(s)&=9s^4\mp6s^3-15s^2\pm24s+16,\\
\Delta_{\pm3}(s)&=81s^4\mp162s^3-1215s^2\pm5832s+11664.
\end{align*}
For large $s$, each discriminant lies strictly between consecutive squares:
\begin{align*}
(3s^2-s-3)^2&<\Delta_{+1}(s)<(3s^2-s-2)^2 &&(s\ge6),\\
(3s^2+s-3)^2&<\Delta_{-1}(s)<(3s^2+s-2)^2 &&(s\ge9),\\
(9s^2-9s-72)^2&<\Delta_{+3}(s)<(9s^2-9s-71)^2 &&(s\ge255),\\
(9s^2+9s-73)^2&<\Delta_{-3}(s)<(9s^2+9s-72)^2 &&(s\ge250).
\end{align*}
The corresponding lower and upper differences are, respectively,
\[
\begin{array}{c|cc}
\varepsilon & \text{lower difference} & \text{upper difference}\\ \hline
+1 & 2s^2+18s+7 & 4s^2-20s-12\\
-1 & 2s^2-18s+7 & 4s^2+20s-12\\
+3 & 4536s+6480 & 18s^2-4554s-6623\\
-3 & 18s^2-4518s+6335 & 4536s-6480
\end{array}
\]
and are positive in the displayed ranges.

The remaining finite ranges admit the following exact certificate. Only reduced fractions are retained, so for $|\varepsilon|=3$ we impose $3\nmid s$.
\[
\begin{array}{c|c|c|c}
\varepsilon & \text{admissible }s & \text{square cases} & \text{corresponding }d\\ \hline
+1 & 2\le s\le5 & s=2,\ \Delta_{+1}=100 & d=-4,6\\
-1 & 2\le s\le8 & s=2,\ \Delta_{-1}=100 & d=2,12\\
+3 & 4\le s\le254,\ 3\nmid s & \text{none} & \text{none}\\
-3 & 4\le s\le249,\ 3\nmid s & s=4,\ \Delta_{-3}=0 & d=8/3
\end{array}
\]
(The additional square cases at $s=6$ for $\varepsilon=\pm3$ represent the non-reduced fractions $\pm3/6=\pm1/2$ and have already appeared in the first two rows.) Hence the only pairs $(d,t)$ with $d\ge2$ an integer and $t\in\Q$ a root of $g_d$ are
\[
(d,t)\in\{(2,-\tfrac12),\ (6,\tfrac12),\ (12,-\tfrac12)\},
\]
which proves Lemma~\ref{lem:rational-roots}.  The exact enumeration is reproduced by the script \path{corank1_exact.py} (Appendix~\ref{app:repro}).

\subsection{Factorisation ranges}\label{app:factorisation-ranges}
The companion script \path{corank1_exact.py} constructs $K_m^{(d)}$ over $\Q$ from its normalised Jacobi-polynomial formula and factors $K_m^{(d)}\pm1$ exactly.  For $3\le d\le300$ with $2\le m\le12$, for $3\le d\le60$ with $13\le m\le24$, and for $d=3$ with $2\le m\le40$, the ranges contain respectively $973$, $132$ and $20$ parity-eligible pairs.  In every case at least one of $K_m^{(d)}+1$ and $K_m^{(d)}-1$ is irreducible over $\Q$.  The same script reproduces the finite discriminant certificate of Appendix~\ref{app:rational-roots} and verifies the companion factorisations at $d=6$ and $d=12$.  This is a computation, not a theorem; it is recorded to delimit the open problem stated in Remark~\ref{rem:what-remains} and Section~\ref{sec:further}.

\section{Reproducibility statement}\label{app:repro}

The proofs are mathematical and independent of software.  The ancillary archive contains pinned requirements, deterministic logs, checksums, and the following exact scripts:
\begin{enumerate}
\footnotesize
\item \path{verify_finite_defect_global.py}, for the degree-two kernel, strength-four spectral identities, nine-point obstruction, Gale coupling, and stability collision minor;
\item \path{corank1_exact.py}, for the factorisation ranges, the finite strength-six certificate, and the companion cubics at $d=6,12$;
\item \path{spherical_design_alias_closure_verify.py}, for the determinant and tangent identities in the generic branch;
\item \path{verify_two_alias_pauli_branch.py}, for the Pauli determinant, harmonic-spinor equations, trace and reality identities, spin-$3/2$ covariance, and canonical phases;
\item \path{verify_higher_dimensional_structure.py}, for the Hilbert vector, forced-socle dimension, and heptagon power sums;
\item \path{verify_finite_defect_all_corank.py}, for the binomial/kernel ledger behind the antipodal bound and the all-corank dimension identities;
\item \path{corank1_residue.py}, for the symbolic Laurent coefficients and exact polynomial-remainder checks in Theorem~\ref{thm:residue};
\item \path{corank1_residue_explore.py}, for the exact evaluation of \eqref{eq:residue-formula} in Remarks~\ref{rem:polygon-check} and~\ref{rem:complementary}.
\end{enumerate}
Running \texttt{python run\_all.py} executes the complete exact suite.  The scripts are reproducibility aids, not substitutes for the proofs.  No numerical optimisation or finite linear-programming result is used as a theorem in this version.

\end{document}